\documentclass[11pt]{article}

\usepackage{amsfonts,amsmath,amsthm,times,fullpage,amssymb}

\usepackage[dvipsnames]{xcolor}

\usepackage[noend]{algpseudocode}
\usepackage[nothing]{algorithm}  

\usepackage{caption}
 
\makeatletter 
\newcommand\suchthat{%
 \@ifstar
  {\mathrel{}\middle|\mathrel{}}
  {\mid}% 
}
\makeatother

\newcommand{\black}[1]{{\color{black}{#1}}}

\newtheorem{theorem}{Theorem} 
\newtheorem{lemma}{Lemma}
\newtheorem{definition}{Definition} 
 
\newtheorem{remark}{Remark}  
\newtheorem{claim}{Claim}  
\newtheorem{proposition}{Proposition}

\newtheorem*{gLLL}{General LLL}

\newtheorem*{lLLL}{Lopsided LLL}

\newtheorem*{atomicity}{Atomicity}
\newtheorem*{amenability}{Amenability}

\newtheorem*{cd}{Potential Causality Digraph}
\newtheorem*{causality}{Potential Causality}

\newcommand{\Indep}{\mathrm{Ind}}

\newcommand{\List}{\mathrm{List}}  
\newcommand{\Versioned}{\mathrm{Ver}}

\begin{document}

\title{Random Walks that Find Perfect Objects\\ and the Lov\'{a}sz Local Lemma}

\author{
Dimitris Achlioptas
\thanks{
Research partially performed at RACTI  and the Department of Informatics and Telecommunications, University of Athens. Research supported by ERC Starting Grant 210743 and an Alfred P. Sloan Research Fellowship.} 
\\ Deparment of Computer Science\\ University of California Santa Cruz
\\
\and 
Fotis Iliopoulos\thanks{ 
Research partially performed at the Department of Electrical and Computer Engineering, National Technical University of Athens. Research supported by ERC Starting Grant 210743.} 
 \\ Deparment of Electrical Engineering and Computer Science\\ University of California Berkeley\\{\small{\tt }}
}

\date{\empty}

\maketitle

\begin{abstract}
We give an algorithmic local lemma by establishing a sufficient condition for the uniform random walk on a directed graph to reach a sink quickly. Our work is inspired by Moser's entropic method proof of the Lov\'{a}sz Local Lemma (LLL) for satisfiability and completely bypasses the Probabilistic Method formulation of the LLL. In particular, our method works when the underlying state space is entirely unstructured. Similarly to Moser's argument, the key point is that the inevitability of reaching a sink is established by bounding the entropy of the walk as a function of time.
\end{abstract}
 
\section{Introduction}

Let $\Omega$ be a (large) set of objects and let $F$ be a collection of subsets of $\Omega$, each subset comprising objects sharing some (negative) feature. We will refer to each subset $f \in F$ as a \emph{flaw} and, following linguistic rather than mathematical convention, say that $f$ is present in $\sigma$ if $f \ni \sigma$. We will say that an object $\sigma \in \Omega$ is \emph{flawless} (perfect) if no flaw is present in $\sigma$. For example, given a CNF formula on $n$ variables with clauses $c_1,c_2,\ldots,c_m$, we can define a flaw for each clause $c_i$, comprising the subset of $\Omega = \{0,1\}^n$ violating $c_i$. 
 
Given $\Omega$ and $F$ we can often prove the \emph{existence} of flawless objects using the Probabilistic Method. Indeed, in many interesting cases this is the only way we know how to do so. To employ the Probabilistic Method we introduce a probability measure on $\Omega$ and consider the collection of (``bad") events $\mathcal{A}$ corresponding to the flaws (one event per flaw). The existence of flawless objects is then equivalent to the intersection of the complements of the bad events having strictly positive probability. Clearly, such positivity always holds if the events in $\mathcal{A}$ are independent and none of them has measure 1. One of the most powerful tools of the Probabilistic Method is the Lov\'{a}sz Local Lemma~(LLL) asserting that such positivity also holds under a condition of limited dependence among the events in $\mathcal{A}$. The idea of the Local Lemma was first circulated by Lov\'{a}sz in the early 1970s in an unpublished note. It was published by Erd\H{os} and Lov\'{a}sz in~\cite{lll}.  The general form below is also due in unpublished form to Lov\'{a}sz and was given by Spencer in~\cite{Spencer77}. 
\begin{gLLL}
Let $\mathcal{A} = \{A_1,A_2, \ldots, A_m \}$ be a set of events and let 
$
D(i)
\subseteq [m] \setminus \{i\}$ denote the set of indices of the dependency set of $A_i$, i.e., $A_i$ is mutually independent of all events in $\mathcal{A} \setminus \{A_i \cup \bigcup_{ j \in 
D(i)}
A_j \}$. If there exist positive real numbers 
$\{\mu_i\}$ such that for all $i \in [m]$,
\begin{equation}
\Pr(A_i) \prod_{j \in \{i\} \cup
D(i)
} (1+\mu_j)  \le \mu_i \enspace , \label{eq:lllcond}
\end{equation}
then the probability that none of the events in $\mathcal{A}$ occurs is  at least $\prod_{i=1}^m 1/(1+\mu_i) > 0$.
\end{gLLL}

\begin{remark}
Condition~\eqref{eq:lllcond} above is equivalent to $\Pr(A_i) \le x_i \prod_{j \in 
D(i)
} (1-x_j)$, where $x_i = \mu_i/(1+\mu_i)$. As we will see the formulation~\eqref{eq:lllcond} facilitates comparisons. 
\end{remark}

In~\cite{LopsTrav}, Erd\H{o}s and Spencer noted that one can replace the LLL's requirement that each bad event is dependent with few other bad events with the weaker requirement that each bad event is \emph{negatively correlated} with few other bad events. That is, for each bad event $A_i$ there should only be few other bad events whose non-occurrence may boost $A_i$'s probability of occurring; the non-occurrence of any subset of the remaining events should leave $A_i$ either unaffected, or make it less likely. A natural setting for the lopsided LLL arises when one seeks a collection of permutations satisfying a set of constraints and considers the uniform measure on them. While the bad events (constraint violations) are now typically densely dependent (as fixing the image of even just one element affects all others), one can often establish sufficient negative correlation among the bad events to apply the lopsided LLL. 

\begin{lLLL}[\cite{LopsTrav}]
Let $\mathcal{A} = \{A_1, A_2,\ldots,A_m\}$ be a set of $m$ events. For each $i \in [m]$, let $\Gamma(i)$ be a subset of $[m] \setminus \{i\}$ such that $\Pr(A_i \suchthat \cap_{j \in S} \overline{A_j}) \le \Pr(A_i)$, for every $S \subseteq [m] \setminus (\Gamma(i) \cup \{i\})$. If there exist positive real numbers $\{\mu_i\}$ such that for all $i \in [m]$,
\begin{equation}\label{eq:LLLL}
\Pr(A_i) \prod_{ j \in \{  i \}  \cup  \Gamma(i) } (1+\mu_j)  \le \mu_i \enspace , 
\end{equation}
then the probability that none of the events in $\mathcal{A}$ occurs is at least $\prod_{i=1}^m 1/(1+\mu_i) > 0$. 
\end{lLLL}

In the context of the general LLL it is natural to define the \emph{dependence} digraph $\mathcal{D}$ of a collection of events $\{A_1,A_2, \ldots, A_m \}$ as having a vertex $v_i$ for each event $A_i$ and an arc $(v_i,v_j)$ iff $j \in D(i)$, noting that there exist systems of events such that $\mathcal{D}$ contains arc $(i,j)$ but not arc $(j,i)$. The \emph{lopsided dependence} digraph is the sparsification $\mathcal{D}_L$ of $\mathcal{D}$ wherein each event $A_i$ points only to the events that may boost it, i.e., the elements of the set $\Gamma(i) \subseteq D(i)$. Let $G$ be the undirected graph that results by ignoring arc direction in $\mathcal{D}_L$. Also, observe that condition~\eqref{eq:LLLL} can be trivially rewritten (expanded) as 
\begin{align}
\Pr(A_i) \sum_{S \subseteq \{  i \} \cup  \Gamma(i) }  \prod_{j \in S} \mu_j \le \mu_i \enspace . \label{eq:unpackster}
\end{align}
Relatively recently, Bissacot et al.~\cite{bissacot2011improvement} improved the lopsided LLL when the graph $G$ is not triangle-free. Specifically, they showed that the conclusion of the lopsided LLL remains valid if the summation in~\eqref{eq:unpackster} is restricted to those sets $S \subseteq \{  i \} \cup  \Gamma(i)$ which are \emph{independent} in $G$.

\subsection{Constructive Versions} 

As one can imagine, after proving that $\Omega$ contains flawless objects via the LLL it is natural to ask if some flawless object can be found efficiently. Making the LLL constructive has been a long quest, starting with the work of Beck~\cite{beck_lll}, with subsequent works of Alon~\cite{alon_lll}, Molloy and Reed~\cite{mike_stoc}, Czumaj and Scheideler~\cite{Czumaj_lll}, Srinivasan~\cite{aravind_08} and others. Each such work established a method for finding flawless objects efficiently, but in all cases under significant additional conditions relative to the LLL. The breakthrough was made by Moser~\cite{moser} who showed that a shockingly simple algorithm nearly matches the LLL condition for \mbox{$k$-CNF} formulas. Very shortly afterwards, Moser and Tardos in a landmark paper~\cite{MT} made the general LLL constructive for all \emph{product} measures over explicitly presented variables. 

Specifically, in the so-called \emph{variable setting} of~\cite{MT}, each event $A_i$ is associated with the set of variables $\mathrm{vbl}(A_i)$ that determine it so that $j \in D(i)$ iff $\mathrm{vbl}(A_i) \cap \mathrm{vbl}(A_j) \ne \emptyset$. Moser and Tardos proved that if the condition~\eqref{eq:lllcond} of the general LLL holds, then repeatedly selecting \emph{any} occurring event $A_i$ (flaw present) and resampling every variable in $\mathrm{vbl}(A_i)$ independently of all others, leads to an elementary event where no event in $\mathcal{A}$ holds (flawless object) after a polynomial number of resamplings.  Pegden~\cite{PegdenIndepen} strengthened the result of Moser and Tardos~\cite{MT} by showing that its conclusion still holds if the condition~\eqref{eq:lllcond} of the general LLL is replaced by~\eqref{eq:unpackster}, where the summation is restricted to independent sets, i.e., under the condition of Bissacot et al.~\cite{bissacot2011improvement} mentioned above. Kolipaka and Szegedy in~\cite{szege_meet} showed that the algorithm of Moser and Tardos, in fact, converges in polynomial time under the criterion of Shearer~\cite{Shearer}, the most generous condition under which $\Pr[\cap_i \overline{A_i}]>0$ for symmetric dependency graphs. As the criterion of Shearer is not efficiently verifiable, Kolipaka, Szegedy and Xu~\cite{CLLL} gave a series of intermediate conditions, between the general LLL and Shearer's criterion, for the algorithm of~\cite{MT} to terminate, most notably the efficiently verifiable Clique LLL. On the other hand, with the notable exception of CNF-SAT,  none of these results applies to the lopsided LLL which remained non-constructive.

Very recently Harris and Srinivasan~\cite{SrinivasanPerm} made the lopsided LLL constructive for the uniform measure on Cartesian products of permutations. Among other results this yielded an efficient algorithm for constructing $n \times n$ Latin Squares when each color appears at most $\Delta \le (27/256)n$ times, matching the best non-constructive bound due to Bissacot et al.~\cite{bissacot2011improvement} (who improved the original  $\Delta \le n/(4\mathrm{e})$ bound of Erd\H{o}s and Spencer~\cite{LopsTrav} by exploiting the local density of the lopsided dependency graph). Harris and Srinivasan~\cite{SrinivasanPerm} pointed out that while the permutation setting is the most common use case, the lopsided LLL has been gainfully applied to other settings~\cite{lu2013quest,mohr2013applications} including hypergraph matchings, set partitions and spanning trees, and asked if their results can be extended beyond permutations. In particular, they left as a canonical open problem whether the results of Dudek, Frieze and Ruci\'{n}ski~\cite{Hamilton1} regarding Hamilton Cycles in edge colored hypergraphs can be made constructive.

\section{A New Framework}

Inspired by the breakthrough  of Moser~\cite{moser} we take a more direct approach to finding flawless objects, bypassing the probabilistic formulation of the existence question. Specifically, we replace the measure on $\Omega$ by a directed graph $D$ on $\Omega$ and we seek flawless objects by taking random walks on $D$. With this in mind, we refer to the elements of $\Omega$ as states. As in Moser's work~\cite{moser}, each state transformation (step of the walk) $\sigma \rightarrow \tau$  will be taken to \emph{address} a flaw present at $\sigma$. Naturally, a step may eradicate other flaws beyond the one addressed but may also introduce new flaws (and, in fact, \mbox{may fail to eradicate the addressed flaw).} By replacing the measure with a directed graph we achieve two main effects:
\begin{itemize}
\item
Both the set of objects $\Omega$ and every flaw $f \subseteq \Omega$ can be entirely \emph{amorphous}. That is, $\Omega$ does not need to have product form $\Omega = D_1 \times \cdots \times D_n$, as in the work of Moser and Tardos~\cite{MT}, or any form of symmetry, as in the work of Harris and Srinivasan~\cite{SrinivasanPerm}. For example, $\Omega$ can be the set of all Hamiltonian cycles of a graph, a set of very high complexity.
\item
The set of transformations for addressing a flaw $f$ can differ \emph{arbitrarily} among the different states $\sigma \in f$, allowing the actions to adapt to the ``environment". This is in sharp contrast with all past algorithmic versions of the LLL where either no or very minimal adaptivity was possible. As we discuss in Section~\ref{sec:disc}, this moves the Local Lemma from the Probabilistic Method squarely within the purview of Algorithm Design.
\end{itemize}

Concretely, for  each $\sigma \in \Omega$, let $U(\sigma) = \{f \in F: \sigma \in f\}$, i.e., $U(\sigma)$ is the set of flaws present in $\sigma$. For each $\sigma \in \Omega$ and $f \in U(\sigma)$ we require a set $A(f,\sigma) \subseteq \Omega$ that must contain at least one element other than $\sigma$ which we refer to as the set of possible \emph{actions} for addressing flaw $f$ in state $\sigma$. To address flaw $f$ in state $\sigma$ we select uniformly at random an element $\tau \in A(f,\sigma)$ and walk to state $\tau$, noting that possibly  $\tau = \sigma \in A(f,\sigma)$. Our main point of departure is that now the set of actions for addressing a flaw $f$ in each state $\sigma$ can depend \emph{arbitrarily} on the state, $\sigma$, itself. 

We represent the set of all possible state transformations as a multi-digraph $D$ on $\Omega$ formed as follows: for each state $\sigma$, for each flaw $f \in U(\sigma)$, for each state  $\tau \in A(f,\sigma)$ place an arc $\sigma \xrightarrow{f} \tau$ in $D$, i.e., an arc labeled by the flaw being addressed. Thus, $D$ may contain pairs of states $\sigma, \tau$ with multiple $\sigma \to \tau$ arcs, each such arc labeled by a different flaw, each such flaw $f$ having the property that moving to $\tau$ is one of the actions for addressing $f$ at $\sigma$, i.e., $\tau \in A(f,\sigma)$. Since we require that the set $A(f,\sigma)$ contains at least one element other than $\sigma$ for every flaw in $U(\sigma)$ we see that a vertex of $D$ is a sink iff it is flawless.

We focus on digraphs satisfying the following condition.
\begin{atomicity}
$D$ is \emph{atomic} if for every flaw $f$ and state $\tau$ there is \emph{at most} one arc incoming to $\tau$ labeled by $f$.
\end{atomicity}

The purpose of atomicity is to capture ``accountability of action". In particular, note that if $D$ is atomic, then every walk on $D$ can be reconstructed from its final state and the sequence of labels on the arcs traversed, as atomicity allows one to trace the walk backwards unambiguously. To our pleasant surprise, in all applications we have considered so far we have found atomicity to be ``a feature not a bug", serving as a very valuable \emph{aid} in the design of flaws and actions, i.e., of algorithms. A fruitful way to think about atomicity is to consider the case where $\Omega$ and $F$ have product structure over a set of variables, e.g., a Constraint Satisfaction Problem. In that case the following suffice to imply atomicity:
\begin{enumerate}
\item\label{satlike}
Each constraint (flaw) forbids exactly \emph{one} joint value assignment to its underlying variables. 
\item\label{nospooky}
Each state transition modifies \emph{only} the variables of the violated constraint (flaw) that it addresses. 
\end{enumerate}

Condition~\ref{satlike} expresses a purely syntactic requirement: compound constraints must be broken down to constituent parts akin of satisfiability constraints. So, for example, to encode graph \mbox{$q$-colorability} we must write $q$ constraints (flaws) per edge, one for each color. Decomposing constraints in this manner enables a uniform treatment at no additional cost. In many cases it is, in fact, strictly advantageous as it affords a more refined accounting of conflict  between constraints.  Condition~\ref{nospooky} on the other hand is a genuine restriction reflecting the idea of ``focusing" introduced by Papadimitriou~\cite{papafocus}, i.e., that every state transformation should be the result of attempting to eradicate some specific flaw. 

To see that Conditions~\ref{satlike} and \ref{nospooky} imply atomicity imagine that there exist arcs $\sigma_1 \xrightarrow{f} \tau$ and $\sigma_2 \xrightarrow{f} \tau$, i.e., two state transformations addressing the same flaw $f$ leading to the same state $\tau$. Since $f$ must be present in both $\sigma_1$ and $\sigma_2$, Condition~\ref{satlike} implies that if $\sigma_1 \neq \sigma_2$, then there exists at least one variable $v$ not bound by $f$ which takes different values in $\sigma_1, \sigma_2$. In that case, though, Condition~\ref{nospooky} implies that $v$ will have the same value before and after each of the two transformations, leading to a contradiction.

Having defined the multi-digraph $D$ on $\Omega$ we will now define a digraph $C$ on the set of flaws $F$, reflecting some of the structure of $D$.

\begin{causality}
For each arc $\sigma \xrightarrow{f}  \tau$ in $D$ and each flaw $g$ present in $\tau$ we say that $f$ causes $g$ if $g = f$ or $g \not\ni \sigma$. If $D$ contains \emph{any} arc in which $f$ causes $g$ we say that $f$ \emph{potentially causes} $g$.
\end{causality}

\begin{cd}
The digraph $C=C(\Omega,F,D)$ of the potential causality relation, i.e., the digraph on $F$ where $f \rightarrow g$ iff $f$ potentially causes $g$, is called the potential causality digraph. The \emph{neighborhood} of a flaw $f$  is $\Gamma(f) =\{g : f \to g \text{  exists in $C$}\}$.
\end{cd}

In the interest of brevity we will call $C$ the causality digraph, instead of the potential causality digraph. It is important to note that $C$ contains an arc $f \to g$ if there exists \emph{even one} state transition aimed at addressing $f$ that causes $g$ to appear in the new state. In that sense, $C$ is a ``pessimistic" estimator of causality (or, alternatively, a lossy compression of $D$). This pessimism is both the strength and the weakness of our approach.
On one hand, it makes it possible to extract results about algorithmic progress without tracking the evolution of state. On the other hand, it only gives good results when $C$ can remain sparse even in the presence of such stringent arc inclusion. We feel that this tension is meaningful: maintaining the sparsity of $C$ requires that the actions for addressing each flaw across different states are \emph{coherent} with respect to the flaws they cause.

Without loss of generality (and to avoid certain trivialities), we can assume that $C$ is strongly connected, implying $\Gamma(f) \ge 1$ for every $f \in F$. 
To see this, let $C_1, \ldots, C_k$ be the strongly connected components of $C$ and consider the DAG with vertices $c_1, \ldots, c_k$, where for $i \neq j$, $c_i$ points to $c_j$ iff there exist  $f \in C_i$ and $g \in C_j$ such that $f \rightarrow g$ exists in $C$. If we have a sufficient condition for finding flawless objects when the causality digraph is strongly connected, then we can take any source vertex $c_i$ in the DAG and repeatedly address flaws in $C_i$ until we reach a state  $\sigma_i \in \Omega$ that is $C_i$-flawless, at which point we remove $c_i$ from the DAG. If $\sigma_i$ has other flaws, we select a new source vertex $c_j$ and repeat the same idea continuing from $\sigma_i$. The actions that will be taken to address flaws in $C_j$ will never introduce flaws in $C_i$ etc.

So far we have not discussed \emph{which} flaw to address in each flawed state, demanding instead a non-empty set of actions $A(f,\sigma)$ for each flaw $f$ present in a state $\sigma$. We discuss the reason for this in Section~\ref{sec:choice}. For now, suffice it to say that we consider algorithms which employ an \emph{arbitrary} ordering $\pi$ of $F$ and in each flawed state $\sigma$ address the greatest flaw according to $\pi$ in a subset of $U(\sigma)$. 
\begin{definition}
If $\pi$ is any ordering of $F$, let $I_{\pi}:2^{F} \rightarrow F$ be the function mapping each subset of $F$ to its greatest element according to $\pi$, with $I_{\pi}(\emptyset) = \emptyset$. We will sometimes abuse notation and for a state $\sigma \in \Omega$, write $I_{\pi}(\sigma)$ for $I_{\pi}(U(\sigma))$ and also write $I$ for $I_{\pi}$ when $\pi$ is clear from context. 
\end{definition}
\begin{definition}
Let $D_{\pi} \subseteq D$ be the result of retaining for each state $\sigma$ only the outgoing arcs with label $I_{\pi}(\sigma)$. 
\end{definition}

The next definition reflects that since actions are selected uniformly, the \emph{number} of actions available to address a flaw, i.e., the breadth of the ``repertoire", is important.
\begin{amenability}
The \emph{amenability} of a flaw $f$ is 
\begin{equation}\label{eq:repertoire}
A_f = \min_{\sigma \in f} |A(f,\sigma)| \enspace .
\end{equation}
\end{amenability}
The amenability of a flaw $f$ will be used to bound from below the amount of randomness consumed every time $f$ is addressed. (The minimum in~\eqref{eq:repertoire} is often inoperative with $|A(f,\sigma)|$ being the same for all $\sigma \in f$.)

\section{Statement of Results}

Our first result concerns the simplest case where, after choosing a single fixed permutation $\pi$ of the flaws, in each flawed state $\sigma$ the algorithm simply addresses the greatest flaw present in $\sigma$ according to $\pi$, i.e., the algorithm is the uniform random walk on $D_{\pi}$. 

\begin{theorem}\label{cor:simple}
If for every flaw $f \in F$,
\[
\sum_{g \in \Gamma(f)} \frac{1}{A_g} < \frac{1}{\mathrm{e}}  \label{touti_dw_xamw} \enspace ,
\]
then for any ordering $\pi$ of $F$ and any $\sigma_1 \in \Omega$, the uniform random walk on $D_{\pi}$ starting at $\sigma_1$ reaches a sink within $(\log_2 | \Omega | +    |U(\sigma_1)|+s)/\delta$ steps with probability at least $1-2^{-s}$, where $\delta  =  1 - \max_{f \in F} \sum_{g \in \Gamma(f) } \frac{ \mathrm{e} }{A_g}$.
\end{theorem}

Theorem~\ref{cor:simple} has three features worth discussing, shared by all our further results below.\medskip

\noindent {\bf Arbitrary initial state.} The fact that $\sigma_1$ can be arbitrary means that any foothold on $\Omega$ suffices to apply the theorem, without needing to be able to sample from $\Omega$ according to some measure. While sampling from $\Omega$ has generally not been an issue in existing applications of the LLL, as we discuss in Section~\ref{sec:disc}, this has only been true precisely because the sets and the measures considered have been highly structured.\smallskip

\noindent {\bf Arbitrary number of flaws.} The running time depends only on the number of flaws present in the initial state, $|U(\sigma_1)|$, not on the total number of flaws $|F|$. This has an implication  analogous to the result of Hauepler, Saha, and Srinivasan~\cite{haeupler_lll} on core events: even when $|F|$ is very large, e.g., super-polynomial in the problem's encoding length, we can still get an efficient algorithm if we can show that $|U(\sigma_1)|$ is small, e.g., by proving that in every state only polynomially many flaws may be present. This feature provides great flexibility in the design of flaws, as demonstrated in one of our applications, presented in Section~\ref{sec:cb}.\smallskip

\noindent {\bf Cutoff phenomenon.}
The bound on the running-time is sharper than a typical high probability bound, being instead akin to a mixing time cutoff bound~\cite{cutoff_Diaconis}, wherein the distance to the stationary distribution drops from near 1 to near 0 in a very small number of steps past a critical point. In our setting, the walk first makes $T_0/\delta$ steps without any guarantee of progress, but from that point on every single step has constant probability of being the last step. While, pragmatically, a high probability bound would be just as useful, the fact that our bound naturally takes this form suggests a potential deeper connection with the theory of Markov chains.\medskip

Theorem~\ref{cor:simple} follows from the following significantly more general result. We present the derivation of Theorem~\ref{cor:simple} from Theorem~\ref{asymmetric} in Section~\ref{sec:BPstate}. Observe the similarity between the condition of Theorem~\ref{asymmetric} and the condition~\eqref{eq:lllcond} of the general LLL with $1/A_f$ replacing $\Pr(A_i)$.
\begin{theorem}[Main result]\label{asymmetric}
If there exist positive real numbers $\{\mu_f\}$ such that for every flaw $f \in F$,
\begin{align}
\frac{1}{A_f}
\prod_{g \in \Gamma(f)} (1+\mu_g) < \mu_f \enspace , \label{eq:mc2}
\end{align}
then for any ordering $\pi$ of $F$ and any $\sigma_1 \in \Omega$, the uniform random walk on $D_{\pi}$ starting from $\sigma_1$ reaches a sink within $(T_0+s)/\delta$ steps with probability at least $1-2^{-s}$, where 
\begin{align*}
\delta & = 1 - \max_{f \in F} \left(\frac{1}{\mu_f A_f} \prod_{g \in \Gamma(f)} (1+\mu_g)\right) > 0 \enspace ,\\
T_0 	& = 
\log_2 | \Omega | +    |U(\sigma_1) |\cdot \log_2 \left(1 + \frac{ \max_{f \in F}(\mu_f A_f)}{  \min_{f \in F} A_f} \right) \enspace  .
\end{align*}
\end{theorem}

\begin{remark}
In applications, typically, $\delta = \Theta(1)$ and $T_0 = O(\log |\Omega| + |U(\sigma_1)| \log|F|)$. 
\end{remark}

\subsection{Dense Neighborhoods}

In a number of applications the subgraph induced by the neighborhood of each flaw in the causality graph contains several arcs. We improve Theorem~\ref{asymmetric} in such settings by employing a \emph{recursive} algorithm. This has  the effect that the flaw addressed in each step depends on the entire trajectory up that point not just the current state, i.e., the walk in non-Markovian. It is for this reason that we required a non-empty set of actions for every flaw present in a state, and why the definition of the causality digraph does not involve flaw choice. Specifically, for any ordering $\pi$ of $F$ and any $\sigma_1 \in \Omega$ the recursive walk is the non-Markovian random walk on $\Omega$ that occurs by invoking procedure {\sc Eliminate} below. Observe that if in line~\ref{a_key_diff} we did not intersect $U(\sigma)$ with $\Gamma(f)$ the recursion would be trivialized and the walk would be the uniform random walk on $D_{\pi}$. This is because the first time any ``while" condition would be satisfied, causing the corresponding recursive call to return, would be when $U(\sigma) = \emptyset$.

\begin{algorithm}\caption*{{\bf Recursive Walk}}
\begin{algorithmic}[1]\label{Recursive}
\Procedure{Eliminate}{}
\State $\sigma \leftarrow \sigma_1$
\While {$U(\sigma) \neq \emptyset$}	
	\State {\sc Address} ($I_{\pi}(\sigma),\sigma$) 
\EndWhile	
\State	\Return $\sigma$ 
\EndProcedure{}{}
\Procedure{Address}{$f,\sigma$}
\State $\sigma \leftarrow$ A uniformly random element of $A(f,\sigma)$  
\While {$B = U(\sigma) \cap \Gamma(f) \neq \emptyset$}   \label{a_key_diff}	\Comment{Note $\,\cap \Gamma(f)$}
	\State {\sc{Address}}($I_{\pi}(B),\sigma$) 				\label{code:act}
\EndWhile  
\EndProcedure  
\end{algorithmic}
\end{algorithm}

\begin{definition}\label{defn:G}
Let $G(F,E)$ be the \emph{undirected} graph on $F$ where $\{f,g\} \in E$ iff both $f \rightarrow g$ and $g \rightarrow f$ exist in the causality digraph $C$. For any $S \subseteq F$, let $\Indep(S) = \{S' \subseteq S : \text{$S'$ is an independent set in $G$}\}$.
\end{definition}

Observe that, trivially, the condition of Theorem~\ref{asymmetric} can be restated as requiring that for every flaw $f \in F$,
\begin{align}
\frac{1}{\mu_f A_f} \sum_{S \subseteq \Gamma(f)}  \prod_{g \in S} \mu_g < 1 \enspace , \label{eq:unpack}
\end{align}
where, throughout, we use the convention that a product devoid of factors equals 1, i.e., $\prod_{x \in\emptyset} f(x) = 1$.

\begin{theorem}\label{olala}
If there exist positive real numbers $\{\mu_f\}$ such that for every flaw $f \in F$,
\begin{align}
\theta_f := \frac{1}{\mu_f A_f} \sum_{S \in \Indep(\Gamma(f))}  \prod_{g \in S} \mu_g < 1 \enspace , \label{eq:mc3}
\end{align}
then for any ordering $\pi$ of $F$ and any $\sigma_1 \in \Omega$, the recursive walk on $D$ starting at $\sigma_1$ reaches a sink within $(T_0+s)/\delta$ steps with probability at least $1-2^{-s}$, where $\delta = 1 - \max_{f \in F} \theta_f$, and
\[
T_0 		= 
\log_2 | \Omega | + \left(\black{\max_{ S \in \Indep(U(\sigma_1))} |S|} \right) \cdot \log_2 \left(1 + \frac{ \max_{f \in F}(\mu_f A_f)}{  \min_{f \in F} A_f} \right) \enspace  .
\]
\end{theorem}

\begin{remark}
Theorem~\ref{olala} improves Theorem~\ref{asymmetric} in that (i) the summation in~\eqref{eq:mc3} is only over the subsets of $\Gamma(f)$ that are independent in $G$, instead of being over all subsets of $\Gamma(f)$ as in~\eqref{eq:unpack}, and (ii) $T_0$ is proportional only to the size of the largest independent subset of $U(\sigma_1)$ rather than to the size of $U(\sigma_1)$.
\end{remark}

\begin{remark}
Theorem~\ref{olala} can be strengthened by introducing for each flaw $f \in F$ a permutation $\pi_f$ of $\Gamma_f$ and replacing $\pi$ with $\pi_f$  in line~\ref{code:act} the of Recursive Walk. With this change, in~\eqref{eq:mc3} it suffices to sum only over $S \subseteq \Gamma(f)$ satisfying the following: if the subgraph of $C$ induced by $S$ contains an arc $g \to h$, then $\pi_{f}(g) \ge \pi_f(h)$. As such a subgraph can not contain both $g \to h$ and $h \to g$ we see that $S \in \Indep(\Gamma(f))$.
\end{remark}

\subsection{A Left-Handed Algorithm}\label{sec:left_result}

While Theorems~\ref{cor:simple}--\ref{olala} do not care about the flaw ordering $\pi$, inspired by the so-called LeftHanded version of the LLL introduced by Pedgen~\cite{PegdenLLLL}, we give a condition under which the flaw order $\pi$ can be chosen in a \emph{provably} beneficial way. This is done by organizing the flaws in an order akin to an elimination sequence. Specifically, the idea is to seek a permutation $\pi$ and a  ``responsibility digraph" $R$, derived from the causality digraph $C$, so as to ``shift responsibility" from flaws failing to satisfy condition~\eqref{eq:mc2} of Theorem~\ref{asymmetric}, to flaws that have slack. 
\begin{definition}
For an ordered set of vertices $v_1 < v_2 < \cdots < v_n$, say that arc $v_i \rightarrow v_j$ is \emph{forward} if $i < j$ and \emph{backward} if $i > j$. Given a causality digraph $C=C(\Omega,F,D)$ and a permutation $\pi$ of $F$ ordering the vertices of $C$, we say that $R$ is a \emph{responsibility digraph for $C$ with respect to $\pi$} if:
\begin{enumerate}
\item 
Every forward arc and self-loop of $C$ exists in $R$.
\item 
If a backward arc $v_j \rightarrow v_i$ of $C$ does not exist in $R$, then for each $k$ such that $v_k \rightarrow v_j$ exists in $R$, $v_k \rightarrow v_i$  exists in $R$ as well.
\end{enumerate}
The neighborhood of a flaw $f$ in a responsibility graph $R$ is $\Gamma_R(f) =  \{ g \in F :  f \rightarrow g \text{ exists in } R  \}$. 
\end{definition}

For any permutation $\pi$ of $F$, any responsibility digraph $R$ with respect to $\pi$, and any $\sigma_1 \in \Omega$, the left-handed walk is the random walk induced on $\Omega$ by modifying the Recursive Walk as follows.

\begin{algorithm}[H]\caption*{{\bf LeftHanded Walk}}
\begin{algorithmic}
\State
\State 
In line~\ref{a_key_diff} of {\bf Recursive Walk} replace $\Gamma$ with $\Gamma_R$.
\end{algorithmic}
\end{algorithm}

\begin{theorem}\label{LeftHanded}
For any permutation $\pi$ of $F$ and any responsibility digraph $R$ with respect to $\pi$, if there exist positive real numbers $\{\mu_f\}$ such that for every flaw $f \in F$,
\[
\theta_f := \frac{1}{\mu_f A_f} \sum_{S \subseteq \Gamma_R(f)}  \prod_{g \in S} \mu_g < 1 \enspace , 
\]
then for any $\sigma_1 \in \Omega$, the lefthanded walk on $D$ starting at $\sigma_1$ reaches a sink within $(T_0+s)/\delta$ steps with probability at least $1-2^{-s}$, where $\delta = 1 - \max_{f \in F} \theta_f$, and 
\[
T_0 		= 
\log_2 | \Omega | +    |U(\sigma_1) |\cdot \log_2 \left(1 + \frac{ \max_{f \in F}(\mu_f A_f)}{  \min_{f \in F} A_f} \right) \enspace  .
\]
\end{theorem}

\begin{remark}
Since the causality digraph $C$ is, trivially, a responsibility graph, Theorem~\ref{LeftHanded} can be seen as a non-Markovian generalization of Theorem~\ref{asymmetric} in which flaw choice is driven by recursion and $R$.
\end{remark}

\section{Comparison to Our Work}\label{sec:disc}

Besides dispensing with the need for $\Omega$ to have product structure (variables) or symmetry (permutations), our setting has two additional benefits.

\subsection{State-dependent Transformations}

The LLL, framed as a result in probability, \emph{begins} with a probability measure on the set of objects $\Omega$. In terms of proving the existence of flawless objects, its value lies in that it delivers strong results even when the measure is chosen without \emph{any} consideration of the flaws (bad events). Indeed, most LLL applications simply employ the uniform measure on $\Omega$, a property that can render the LLL indistinguishable from magic. It is worth noting that in the presence of variables, the uniform measure is nothing but the product measure generated by sampling each variable according to the uniform measure on its domain.

All algorithmic versions of the LLL up to now can be seen as walks on $\Omega$ \emph{constrained} by the measure. For product measures, i.e., in the setting of Moser and Tardos~\cite{MT}, this means that the only transformation allowed is resampling all variables of a bad event, with each variable resampled independently of all others, using the same distribution \emph{every} time the variable is resampled, i.e., obliviously to the current state. The partial resampling framework of Harris and Srinivasan~\cite{partial_resampling} refines this to allow resampling a subset of an event's variables, but again only independently of one another and obliviously to the current state. Similarly, for the uniform measure on permutations~\cite{SrinivasanPerm}: the permuted elements whose images form a violated constraint must be reshuffled in a very specific and state-oblivious way, mandated by consistency with the uniform measure.

In contrast, our framework dispenses with the measure on $\Omega$ altogether allowing the set of transformations for addressing each flaw to depend \emph{arbitrarily} on the current state. This has three distinct effects:
\begin{itemize}
\item
It allows us to deal with settings in which both the set of objects $\Omega$ and the set of flaws are \emph{amorphous}, as in the case of rainbow Hamilton cycles and rainbow perfect matchings, something not possible with any previous algorithmic LLL results.
\item
In the case of permutations, where some structure is present, we derive the same main results as~\cite{SrinivasanPerm} with dramatically simpler proofs. Moreover, we have far greater freedom in the choice of algorithms since there is no constraint imposed by some measure.
\item
Finally, for the variable setting of Moser and Tardos~\cite{MT} we gain ``adaptivity to state". This allows us to address one of the oldest and most vexing concerns about the LLL (see the survey of Szegedy~\cite{mario_survey}), exemplified by the LLL's inability to establish the elementary fact that a graph with maximum degree $\Delta$ can be colored with $q=\Delta + 1$ colors. Specifically, imagine that to recolor a monochromatic edge $e$ we select an endpoint $v$ of $e$ arbitrarily and assign a new color $c$ to $v$. When the choice of $c$ must be uniform among \emph{all} colors, as mandated when using the uniform measure in the variable setting, the obliviousness of the choice necessitates the use of a large number of colors relative to $\Delta$ in order for new violations to become sufficiently rare for the method to terminate. Specifically, the LLL can only work when $q > \mathrm{e}\Delta$. On the other hand, in our setting, the color $c$ can be selected uniformly among the \emph{available} colors for $v$, i.e., the colors not appearing in $v$'s neighborhood, by taking the set of actions $A(f,\sigma)$ to be precisely the set of states that result by assigning available colors to $v$ in $\sigma$. Thus, as soon as $q \ge \Delta+1$, the causality digraph becomes empty and rapid termination follows trivially.
\end{itemize}

\subsection{Dependencies vs. Actions}

Unlike the variable setting of Moser and Tardos~\cite{MT} where the dependency relation between events is symmetric, our causality relation, similarly to the lopsided LLL, is not.  We consider asymmetry a significant structural feature of our work since, as is well-known~\cite{mario_survey}, the directed setting is strictly stronger than the undirected setting. For example,  there exist systems of events for which there exists a \emph{lopsided dependence} digraph sparser than any undirected dependence graph. Moreover, asymmetry is essential in our development of structured clause choice in the left-handed version of our theorem. 

At a high level, our results capture the directedness of the lopsided LLL, but with the far more flexible causality digraph replacing the lopsided dependence digraph. Concretely, our framework replaces the limited negative dependence condition of the lopsided LLL, which can be highly non-trivial to establish~\cite{lu2013quest}, with limited causality under atomicity, a condition that is both significantly less restrictive and far easier to check. Moreover, as mentioned earlier and to our pleasant surprise, in all applications we have considered so far we have found atomicity to be a very valuable \emph{aid} in the design of flaws and actions.

For example, in Section~\ref{sec:ham} we give the first efficient algorithm for finding rainbow Hamilton cycles in hypergraphs, as guaranteed to exist by the non-constructive results of~\cite{Hamilton1,Hamilton2}. When we tried to determine flaws and actions for this setting, to our delight we realized that we could just use one of the main technical propositions of~\cite{Hamilton2}, as it is equivalent to proving that for each flaw $f$ and $\sigma \in f$ there exists a set of actions such that the corresponding digraph $D$ is atomic. As~\cite{Hamilton2} is completely independent of our work we consider this ``coincidence" a nice testament to the naturalness of atomicity. 

In a different direction, in Section~\ref{sec:cb} we give an application regarding the Color-Blind index of Graphs. That setting highlights the importance of the directness of the causality graph demonstrates how directedness readily enables the formulation of ``obvious" flaws and actions. Finally,  our Theorem~\ref{olala} combines the benefits of directedness with the improvement of Bissacot et al., by restricting the summation to independent sets. For example, in Section~\ref{sec:latin} we show how Theorem~\ref{olala} allows us to also give an efficient algorithm for Latin Transversals matching the $\Delta \le (27/256)n$ bound of~\cite{SrinivasanPerm}. (Theorem~\ref{olala} can also benefit the application to rainbow matchings in Section~\ref{sec:match} but we chose to use Theorem~\ref{cor:simple} to keep the exposition simple).

While the main contribution of our framework lies in providing freedom in the design of the set of actions for addressing each flaw in each state (and thus going beyond the LLL), its main limitation is that we are restricted in performing uniform random walks in the corresponding directed graph $D$. That means, for example, that  our framework does not capture applications of the variable setting in which the product measure is not uniform over the domain of each variable, while the variable setting of Moser and Tardos~\cite{MT} captures these cases. We leave closing this gap as future work.

\subsection{Flaw Selection}\label{sec:choice}

As mentioned earlier, in Theorems~\ref{cor:simple}--\ref{olala} the necessary condition is independent of the flaw order $\pi$ and, therefore, if the condition is met the algorithm reaches a sink quickly for \emph{every} permutation $\pi$. As this is an unnecessarily luxurious conclusion, it is natural to try to sharpen the results by selecting the flaw order $\pi$ \emph{first}, so that the causality digraph is the image of the (much) sparser $D_{\pi}$ instead of $D$. However, since an arc $f \rightarrow g$ will exist in the causality digraph $C$ as long as there is even one transition addressing $f$ that causes $g$ in $D_{\pi}$, it is not at all clear that sparsifying $D$ using a generic $\pi$ helps significantly. At the same time, if there exists a ``special" $\pi$ that does help significantly, coming up with it is non-trivial. For example, in the setting of satisfiability, if $f,g$ are clauses that share variable $v$ with opposite signs, then not having the arc $f \to g$ in $C$ requires either that addressing $f$ should never involve flipping $v$, cutting $A_f$ by half, or finding a permutation of the clauses such that in \emph{every} state in which $f$ is the greatest violated clause, $g$ is satisfied by some variable other than $v$. The only non-trivial case we know where the latter can be done is when $F$ is satisfiable by the pure literal heuristic. 

As far as we know, the method by which a bad event (flaw) is selected in each step does not affect the performance of any of the algorithmic extensions of the LLL even though in the setting of~\cite{MT} this choice can be arbitrary. The only use we know of this freedom lies in enabling parallelization when $\Omega$ is a structured set, i.e., when $\Omega$ has product structure\cite{MT,szege_meet,Haeupler}, or it is a set of permutations~\cite{SrinivasanPerm}. Since we allow $\Omega$ to be completely amorphous, it is not readily clear how to approach parallelization in our setting.

Finally, we note that flaw choice in our framework is not really restricted to using a single permutation. For example, in the non-recursive setting, before beginning the walk we can select an arbitrary infinite sequence of permutations $\pi_1, \pi_2, \ldots$ of $F$ and in the $i$-th step of the walk address the greatest flaw present according to $\pi_i$. If $\pi_1 = \pi_2 = \cdots$ we are back to the single-permutation setting, while if, for example, each $\pi_i$ is an independent uniformly random permutation, the algorithm addresses a uniformly random flaw present in each step. At the same time, we must make clear that our framework does \emph{not} accomodate arbitrary flaw selection functions and, in fact, we do not see how to extend it beyond permutation-based choice.  To keep the presentation of our results uniform (and compact) we have stated both Theorems~\ref{asymmetric} and~\ref{LeftHanded} in terms of a single permutation. We do point out the one place in our proofs that changes (trivially) to handle multiple permutations. 

\section{Mapping Bad Trajectories to Forests}\label{sec:forests}

We prove Theorems~\ref{asymmetric}--\ref{LeftHanded} in three parts. In the first part, carried out in this section, we show how to represent each sequence of $t$ steps that does not reach a sink as a forest with $t$ vertices, where the forests have different characteristics for each of the walks of Theorems~\ref{asymmetric}--\ref{LeftHanded}. Then, in Section~\ref{sec:BPstate}, we state a general lemma for bounding the running time of different walks in terms of properties of their corresponding forests and show how it readily implies each of Theorems~\ref{asymmetric}--\ref{LeftHanded}. Finally, in Section~\ref{sec:bpproof} we prove the lemma itself. In a first reading the reader may want to skip Section~\ref{left_forests} (and, perhaps also Section~\ref{rec_forests}). The sections can be read later, in order, after the material of Section~\ref{TheProof} has been absorbed.

In the following to lighten notation we will assume that $\sigma_1 \in \Omega$ is fixed but arbitrary.

\begin{definition} A walk  $ \Sigma =  \sigma_1 \xrightarrow{w_1} \sigma_2 \xrightarrow{w_2} \sigma_3 \cdots \sigma_{t} \xrightarrow{w_t}\sigma_{t+1}$  is called a $t$-trajectory.  A $t$-trajectory is \emph{bad} if it only goes through flawed states. Let $\mathrm{Bad}(t)$ be the set of bad $t$-trajectories starting at $\sigma_1$.
\end{definition}

Our first step is the same as Moser's~\cite{moser}, generalized to the notion of atomicity. It amounts to defining an almost-1-to-1 map from bad $t$-trajectories to sequences of $t$ flaws. While the map is not 1-1, crucially, it becomes 1-1 with the addition of a piece of information whose size is \emph{independent} of $t$. 
\begin{definition}
If $ \Sigma =  \sigma_1 \xrightarrow{w_1} \sigma_2 \xrightarrow{w_2}\sigma_3 \cdots \sigma_{t} \xrightarrow{w_t}\sigma_{t+1}$ is a bad $t$-trajectory, the sequence $W(\Sigma) = w_1, w_2, \ldots,w_t$, i.e., the sequence of flaws labeling the arcs $\Sigma$, is the \emph{witness} of $\Sigma$.  
\end{definition}
\begin{claim}\label{invert}
If $D$ is atomic, then the map from bad $t$-trajectories $\Sigma \to \langle W(\Sigma), \sigma_{t+1} \rangle$ is one-to-one.
\end{claim}
\begin{proof}
The atomicity of  $D$ implies that $\sigma_{t}$ is the unique state in $\Omega$ with an arc $\sigma_{t} \xrightarrow{w_t} \sigma_{t+1}$. Etc.
\end{proof}
Thus, $|\mathrm{Bad}(t)|$ is bounded by the number of possible witness $t$-sequences multiplied by $|\Omega|$.

\subsection{Forests of the Uniform Walk (Theorem~\ref{asymmetric})} \label{TheProof}

Recall that for any ordering $\pi$ of the flaws we denote by $D_{\pi}$ the digraph that results from $D$ if at each state $\sigma$ we only retain the outgoing arcs labeled by $I_{\pi}(\sigma)=I(\sigma)$. To analyze the uniform random walk on $D_{\pi}$ we will represent witnesses as sequences of sets reflecting causality. 

Let $B_i$ be the set of flaws ``introduced" by the $i$-th step of the walk, where a flaw $f$ is said to ``introduce itself" if it remains present after an action from $A(f, \sigma_{i})$ is taken. Formally,
\begin{definition}
Let  $B_0 = U(\sigma_1)$. For $1 \le i\le t-1$, let $B_i = U(\sigma_{i+1}) \setminus ( U (\sigma_i) \setminus I(\sigma_i) )$. 
\end{definition}

Let $B_i^* \subseteq B_i$ comprise those flaws addressed in the course of the trajectory. Thus, $B_i^* = B_i \setminus \{O_i \cup N_i \}$, where $O_i$ comprises any flaws in $B_i$ that were eradicated ``collaterally" by an action taken to address some other flaw, and $N_i$ comprises any flaws in $B_i$ that remained present in every subsequent state after their introduction without being addressed. Formally,
\begin{definition}\label{def:bs}
The \emph{Break Sequence} of a bad $t$-trajectory is $B_0^*, B_1^*, \ldots, B_{t-1}^*$, where for $0 \le i \le t-1$,
\begin{align*}
O_i 		& 	= 	\{f \in B_i \mid  \exists j \in [i+1,t] :   f \notin U(\sigma_{j+1})  \wedge  \forall \ell \in [i+1,j]:   f \ne w_{\ell} \} \\
N_i 		&  	= 	\{f \in B_i \mid \forall j \in [i+1, t] :    f \in 	 U(\sigma_{j+1})  \wedge  \forall \ell \in [i+1,t]:   f \ne w_{\ell} \} \\
B_i^*  	&  = B_i \setminus \{O_i \cup N_i \} \enspace .
\end{align*}

\end{definition}

Given $B_0^*,B_1^*,\ldots,B_{i-1}^*$ we can determine $w_1, w_2, \ldots, w_i$ inductively, as follows. Define $E_1 = B_0^*$, while for $i \ge 1$,
\begin{equation}\label{eq:ri}
E_{i+1} = (E_{i} - w_i) \cup B_i^*  \enspace .
\end{equation}
By construction, the set $E_i \subseteq U(\sigma_i)$ is guaranteed to contain $w_i = I(\sigma_i)=I(U(\sigma_i))$. Since $I = I_{\pi}$ returns\footnote{If instead of $\pi$ we had a sequence of permutations $\pi_1, \pi_2, \ldots$, we would simply use $I_{\pi_i}$ to determine $w_i$ from $E_i$.} the greatest flaw in its input according to $\pi$, it must be that $I_{\pi}(E_i) = w_i$. We note that this is the only place we ever make use of the fact that the function $I$ is derived by an ordering of the flaws, thus guaranteeing that for every $f \in F$ and $S \subseteq F$, if $I(S) \neq f$ then $I(S\setminus f) = I(S)$. 

We next give another 1-to-1 map, mapping each Break Sequence to a vertex-labelled rooted forest. Specifically, the \emph{Break Forest} of a bad $t$-trajectory $\Sigma$ has $|B_0^*|$ trees and $t$ vertices, each vertex labelled by an element of $W(\Sigma)$. To construct it we first lay down $|B_0^*|$ vertices as roots and then process the sets $B_1^*, B_2^*, \ldots$ in order, each set becoming the progeny of an already existing vertex (empty sets, thus, giving rise to leaves).
\begin{algorithm}\caption*{{\bf Break Forest Construction}}
\begin{algorithmic}[1]
\State Lay down $|B_0^*|$ vertices, each labelled by a different element of $B_0^*$, and let $V$ consist of these vertices
\For {$i=1$ to $t-1$} 
\State Let $v_i$ be the vertex in $V_i$ with greatest label according to $\pi$
\State Add $|B_i^*|$ children to $v_i$, each labelled by a different element of $B_i^*$
\State Remove $v_i$ from $V$; add to $V$ the children of $v_i$.
\EndFor
\end{algorithmic}
\end{algorithm}

Observe that even though neither the trees, nor the nodes inside each tree of the Break Forest are ordered, we can still reconstruct $W(\Sigma)$ since the set of labels of the vertices in $V_i$ equals $E_i$ for all $0 \le i\le t-1$.

\subsection{Forests of the Recursive Walk (Theorem~\ref{olala})}\label{rec_forests}

We will represent each bad $t$-trajectory, $\Sigma$, of the Recursive Walk as a vertex-labeled unordered rooted forest, having one tree per invocation of procedure {\sc{address}} by procedure {\sc{eliminate}}. Specifically, to construct the \emph{Recursive Forest} $\phi=\phi(\Sigma)$ we add a root vertex per invocation of {\sc{address}} by {\sc{eliminate}} and one child to every vertex for each (recursive) invocation of {\sc{address}} that it makes. As each vertex corresponds to an invocation of {\sc{address}} (step of the walk) it is labeled by the invocation's flaw-argument. Observe now that (the invocations of {\sc{address}} corresponding to) both the roots of the trees and the children of each vertex appear in $\Sigma$ in their order according to $\pi$. Thus, given the unordered rooted forest $\phi(\Sigma)$ we can order its trees and the progeny of each vertex according to $\pi$ and recover $W(\Sigma)$ as the sequence of vertex labels in the preorder traversal of the resulting ordered rooted forest.

Recall the definition of graph $G$ on $F$ from Definition~\ref{defn:G}. We will prove that the flaws labeling the roots of a Recursive Forest are independent in $G$ and that the same is true for the flaws labelling the progeny of every vertex of the forest. To do this we first prove the following.

\begin{proposition}\label{prop:rid}
If {\sc{address}}($f,\sigma$) returns at state $\tau$, then $U(\tau) \subseteq U(\sigma) \setminus (\Gamma(f) \cup \{f\})$. 
\end{proposition}
\begin{proof}
Let $\sigma'$ be any state subsequent to the {\sc{address}}($f,\sigma$) invocation.
If any flaw in $U(\sigma) \cap \Gamma(f)$ is present at $\sigma'$, the ``while" condition in line~\ref{a_key_diff} of the Recursive Walk prevents {\sc{address}}($f,\sigma$) from returning. On the other hand, if $h \in \Gamma(f) \setminus U(\sigma)$ is present in $\sigma'$, then there must have existed an invocation {\sc{address}}($g, \sigma''$), subsequent to invocation {\sc{address}}($f, \sigma$), wherein addressing $g$ caused $h$. Consider the last such invocation. If $\sigma'''$ is the state when this invocation returns, then $h \not\in U(\sigma''')$, for otherwise the invocation could not have returned, and by the choice of invocation, $h$ is not present in any subsequent state between $\sigma'''$ and $\tau$.  
\end{proof}

Let $(f_i,\sigma_i)$ denote the argument of the $i$-th invocation of {\sc{address}} by {\sc{eliminate}}. By Proposition~\ref{prop:rid}, $\{U(\sigma_i)\}_{i \geq 1}$ is a decreasing sequence of sets. Thus, the claim regarding the root labels follows trivially: for each $i \ge 1$, the flaws in $\Gamma(f_i) \cup f_i$ are not present in $\sigma_{i+1}$ and, therefore, are not present in $U(\sigma_j)$, for any $j \ge i+1$. The proof for the children of each node is essentially identical. If a node corresponding to an invocation {\sc{address($f, \sigma$)}} has children corresponding to (recursive) invocations with arguments $\{(g_i,\sigma_i)\}$, then the sequence of sets $\{U(\sigma_i)\}_{i \geq 1}$ is decreasing. Thus, the flaws in $\Gamma(g_i) \cup g_i$ are not present in $\sigma_{i+1}$ and, therefore, not present in $U(\sigma_j)$, for any $j \ge i+1$.

\subsection{Forests of the LeftHanded Walk (Theorem~\ref{LeftHanded})}\label{left_forests}

Recall that $\pi$ is an arbitrary permutation of $F$ and that the Lefthanded Walk is the Recursive Walk modified by replacing $\Gamma(f)$ with $\Gamma_R(f)$ in line~\ref{a_key_diff}, where $R$ is a responsibility graph for $D$ with respect to $\pi$. We map the bad trajectories of the LeftHanded Walk into vertex-labeled unordered rooted forests, exactly as we did for the bad trajectories of the Recursive Walk, i.e., one tree per invocation of {\sc{address}} by {\sc{eliminate}}, one child per recursive invocation of {\sc{address}}, all vertices labeled by the flaw-argument of the invocation. The challenge for the Lefthanded Walk is to prove that the labels of the roots are distinct and, similarly, that the labels of the children of each node are distinct. (For Break Forests both properties were true automatically; for Recursive Forests we established the stronger property that each of these sets of flaws is independent). To do this we first prove the following analogue of Proposition~\ref{prop:rid}. 

\begin{definition}
Let $S_f$ denote the set of flaws strictly greater than $f$ according to $\pi$. For a state $\sigma$ and a flaw $f \in U(\sigma)$, let $W(\sigma,f) = U(\sigma) \cap S_f$.  
\end{definition}

\begin{proposition}\label{sweep}
If {\sc{address}}$(f,\sigma)$ returns at state $\tau$, then $\tau \not\in f$ and $W(\tau,f) \subseteq W(\sigma,f)$.
\end{proposition}

\begin{proof}
The execution of {\sc{address}}$(f,\sigma)$ generates a recursion tree, each node labeled by its flaw-argument. Thus, the root is labelled by $f$ and each child of the root is labelled by a flaw in $\Gamma_{R}(f)$. Let $S^{+}_f = S_f \cup \{ f \}$. For a state $\omega$, let $Q(f,\omega)$ be the set of flaws in  $S^{+}_f  \setminus \Gamma_R(f)$ that are present in $\omega$. We claim that if $g \in \Gamma_R(f)$ and {\sc{address}}$(g,\omega)$ terminates at $\omega'$, then $Q(f,\omega') \subseteq Q(f,\omega)$. This suffices to prove the lemma as:
\begin{itemize}
\item
By the claim, any flaw in $Q(f,\tau) \setminus Q(f,\sigma)$ must be introduced by the action $\sigma \xrightarrow{f} \sigma'$ taken by the original invocation {\sc{address}}$(f,\sigma)$. Thus, $Q(f,\tau) \subseteq Q(f,\sigma')$.
\item
All flaws in $S^{+}_f$ introduced by $\sigma \xrightarrow{f} \sigma'$ are in $\Gamma_R(f)$, since $R$ contains all forward edges and self-loops of $C$. Thus, $Q(f,\sigma') \subseteq Q(f,\sigma)$. In particular, $f$ can only be present in $\sigma'$ if $f \in \Gamma_R(f)$.
\item
No flaw in $\Gamma_R(f)$ can be present in $\tau$ since {\sc{address}}$(f,\sigma)$ returned at $\tau$.
\end{itemize}

To prove the claim, consider the recursion tree of {\sc{address}}$(g,\omega)$. If $h \in Q(f,\omega')$ and $h \notin Q(f,\omega)$, then there has to be a path $g_1=g, g_2, \ldots, g_{i}$ from the root of the recursion tree of  {\sc{address}}$(g,\omega)$ to a node $g_{i}$ such that $h \in \Gamma (g_{i})$ but $h \notin \Gamma_{R}(g_{j})$ for each $j \in [i]$. To see this, notice that since $h$ was absent in $\omega$ but is present in $\omega'$, it must have been introduced by some flaw $g_i$ addressed during the execution of {\sc{address}}$(g,\omega)$. But if $h$ belonged in the neighborhood with respect to $R$ of any of the flaws on the path from the root to $g_i$, the algorithm would have not terminated. However, such a path can not exist, as it would require all of the following to be true, violating the definition of responsibility digraphs (let $g_0=f$ for notational convenience): (i) $h \in \Gamma(g_i)$, (ii) $h \notin \Gamma_{R}(g_i)$, (iii) $g_i \in \Gamma_R(g_{i-1})$,  and (iv) $ h \notin \Gamma_R(g_{i-1})$.
\end{proof}

To establish the distinctness of the root labels, observe that each time procedure {\sc{eliminate}} is invoked at a state $\sigma$, by definition of $I_{\pi}$, we have   $W(\sigma, (I_{\pi}(\sigma)) = \emptyset$. By Proposition~\ref{sweep}, if the invocation returns at state $\tau$, then neither $I_{\pi}(\sigma)$ nor any greater flaws are present in $\tau$.  Therefore, {\sc{eliminate}}  invokes  {\sc{address}} at most once for each $f \in F$. To see the distinctness of the labels of the children of each node, consider an invocation of {\sc{address($f, \sigma$)}}. Whenever this invocation recursively invokes {\sc{address($g, \sigma'$)}}, where $g \in \Gamma_R(f)$, by definition of $I_{\pi}$, every flaw in $S_g \cap \Gamma_{R}(f)$  is absent from $\sigma'$. By Proposition~\ref{sweep}, whenever each such invocation returns neither $g$ nor any of the flaws in $S_g \cap \Gamma_{R}(f)$ are present implying that {\sc{address($f, \sigma$)}} invokes {\sc{address($g, \sigma'$)}} at most once for each $g \in \Gamma_R(f)$.

\section{A General Forest Lemma and Proof of Theorems~\ref{cor:simple}--\ref{LeftHanded}}\label{sec:BPstate}

Recall that we are considering random walks on the multi-digraph $D$ on $\Omega$ which has an arc $\sigma \xrightarrow{f} \tau$ for each $\sigma \in \Omega$, flaw $f \ni \sigma$, and $\tau \in A(f,\sigma)$. Recall also that the different walks of Theorems~\ref{asymmetric}--\ref{LeftHanded} differ only on \emph{which} flaw to address among those present in the current state $\sigma$. Having chosen to address a flaw $f \ni \sigma$, all three walks proceed in the exact same manner, selecting the next state  $\tau \in A(f,\sigma)$ \emph{uniformly} at random. In Section~\ref{sec:forests} we saw how to map the bad trajectories of the different walks into unordered rooted forests so that given a trajectory's forest and final state we can reconstruct it.  

Next we will formulate and prove a general tool for bounding the running time of different walks on $D$.

\begin{lemma}[Witness Forests]\label{lem:master}
Consider \emph{any} random walk on $D$ which (i) in every flawed state $\sigma$, after choosing (arbitrarily) which flaw $f \ni \sigma$ to address, selects the next state $\tau \in A(f,\sigma)$ uniformly at random, and (ii) whose bad trajectories can be mapped into unordered rooted forests satisfying the following properties, so that given a trajectory's forest we can reconstruct the sequence of flaws addressed along the trajectory:
\begin{enumerate}
\item 
Each vertex of the forest is labeled by a flaw $f \in F$.\label{cond:label}
\item 
The flaws labeling the roots of the forest are distinct and, as a set, belong in the set  $\mathrm{Roots}(\sigma_1) \subseteq 2^{F}$.\label{cond:roots}
\item 
The flaws labeling the children of each vertex are distinct.\label{cond:distinct}
\item
If a vertex is labelled by flaw $f$, the labels of its children, as a set, belong in the set $\mathrm{List}(f) \subseteq 2^{\Gamma(f)}$.\label{cond:list}
\end{enumerate}
If there exist positive real numbers $\{\mu_f\}$ such that for every flaw $f \in F$,
\[
\theta_f := \frac{1}{\mu_f A_f} \sum_{S \in \List(f)}  \prod_{g \in S} \mu_g < 1 \enspace , 
\]
then for any $\sigma_1 \in \Omega$, a walk started at $\sigma_1$ reaches a sink within $(T_0+s)/\delta$ steps with probability at least $1-2^{-s}$, where $\delta = 1 - \max_{f \in F} \theta_f$, and
\[
T_0 		= 
\log_2 | \Omega | + \left(\max_{ S \in \mathrm{Roots}(\sigma_1)} |S|\right) \cdot \log_2 \left(1 + \frac{ \max_{f \in F}(\mu_f A_f)}{  \min_{f \in F} A_f} \right) \enspace  .
\]
\end{lemma}

\subsection{Proof of Theorems~\ref{asymmetric}--\ref{LeftHanded} from Lemma~\ref{lem:master}}

Theorem~\ref{asymmetric} follows immediately by observing that Break Forests trivially satisfy the conditions of Lemma~\ref{lem:master} with $\mathrm{Roots}(\sigma_1) = 2^{U(\sigma_1)}$ and $\mathrm{List}(f) = 2^{\Gamma(f)}$.  Theorem~\ref{olala} follows by observing that Recursive Forests satisfy the conditions with $\mathrm{Roots}(\sigma_1) = \Indep(U(\sigma_1))$ and $\mathrm{List}(f) = \Indep(\Gamma(f))$.   Theorem~\ref{LeftHanded} follows by observing that LeftHanded Forests satisfy the conditions  with $\mathrm{Roots}(\sigma_1) = 2^{U(\sigma_1)}$ and $\mathrm{List}(f) = 2^{\Gamma_R(f)}$. 

\subsection{Proof of Theorem~\ref{cor:simple} from Theorem~\ref{asymmetric}}\label{sec:proofcor}

Let $Z \ge 1$ be the least common multiple of the integers $\{ A_f : f \in F    \}$. Let 
$
d := \max_{ f \in F}  \sum_{ g \in  \Gamma(f)   } \frac{Z}{A_g}
$.
Observe that $d \ge \max_{f \in F } | \Gamma(f) |$ since $|\Gamma(f)| \ge 1$ for every $f \in F$, and that for any set $S \subseteq \Gamma(f)$,
\begin{equation}\label{noversions}
\prod_{g \in S} \frac{Z}{A_g} 
= 
\prod_{g \in S}{\binom{Z/A_g}{1}} \le 
\binom{\sum_{g \in \Gamma(f)} Z/A_g}{|S|}
\le  \binom{d}{|S|} \enspace . 
\end{equation}

Taking $\mu_f  =  Z/(dA_f)>0$ and invoking~\eqref{noversions} we see that the hypothesis of Theorem~\ref{cor:simple} implies
\[
\frac{1}{ \mu_f A_f}  \prod_{g \in \Gamma(f)} (1+\mu_g)   =  
\frac{d}{Z}   \sum_{S \subseteq \Gamma(f)}  \prod_{g \in S }  \frac{Z}{d A_g} = \frac{1-\delta}{\mathrm{e}}
\sum_{i = 0}^{d}   \left( \frac{1}{d} \right)^{i} \binom{d}{i}  = \frac{1-\delta}{\mathrm{e}} \left( 1 + \frac{1}{d} \right)^ d \le 1 -\delta \enspace. 
\]
Regarding the running time, observe that $\max_{f \in F}  (\mu_f A_f ) / \min_{f \in F} A_f \le (Z/d) \le 1$ since $|\Gamma(f)| \ge 1$.

\section{ Proof of Lemma~\ref{lem:master}}\label{sec:bpproof}

\subsection{Versions of Flaws}

Recall that we are considering random walks on the multidigraph $D$ on $\Omega$ which has an arc $\sigma \xrightarrow{f} \tau$ for each $\sigma \in \Omega$, flaw $f \ni \sigma$, and $\tau \in A(f,\sigma)$. For the proof it will be convenient to transform $D$ to another multidigraph $D^*$ as described below. The transformation is trivial from an algorithmic point of view, but helps with the eventual counting. Let $Z\ge 1$ be the least common multiple of the integers $\{A_f : f \in F \}$. 

To form the multidigraph $D^*$ we replace each arc $\sigma \xrightarrow{f} \tau$ in $D$ with $Z/A_f$ arcs from $\sigma$ to $\tau$, carrying labels $f_1, f_2, \ldots, f_{Z/A_f}$. We refer to each such label as a \emph{version} of flaw $f$. To move in $D^*$ from a state $\sigma$, exactly as in $D$, the walk first determines which flaw $f \ni \sigma$ to address and then chooses $\tau \in A(f,\sigma)$ uniformly at random. The only difference is that having done so, now the walk also consumes an additional amount of randomness to ``choose a version" of $f$, i.e., to chose one of the $Z/A_{f}$ arcs from $\sigma$ to $\tau$. Thus, the probability distribution on sequences of states of the walk in $D^*$ is identical to the one in $D$  (indeed the two walks can be coupled so that the sequences are always equal).

\begin{definition}
A trajectory $\Sigma = \sigma_1 \xrightarrow{ w_1 } \sigma_2 \xrightarrow{w_2 } \sigma_3 \cdots \sigma_{t} \xrightarrow{w_t } \sigma_{t+1} $  on $D^*$ where  $w_i $ is a version of the flaw addressed at the $i$-step is called a versioned $t$-trajectory. A versioned $t$-trajectory is \emph{bad} if it only goes only through flawed states. Let $\mathrm{VerBad}(t)$ be the set of all versioned bad $t$-trajectories.
\end{definition}

Observe that to move in $D^*$ from any flawed state $\sigma$ to the next state the walk must select among 
\begin{equation}\label{alkis_is_gone}
|A \left(f,\sigma \right)| \cdot \frac{Z}{A_{f}}  \ge Z
\end{equation}
possibilities, implying that every versioned bad $t$-trajectory has probability at most $1/Z^t$. Having a uniform upper bound of probability as a function of length is precisely why we introduced versioned flaws.

To prove Lemma~\ref{lem:master} we will give $T_0 = T_0 \left(|\Omega|, U(\sigma_1), \{ (A_f,\mu_f): f \in F \} \right)$ such that the probability that a versioned $(T_0+s)$-trajectory on $D^*$ is bad is exponentially small in $s$. Per our discussion above to prove this it suffices to prove that $|\mathrm{VerBad}(t)|/Z^{t}$ is exponentially small in $s$ for $t=T_0+s$. Since $D$ is atomic, we can reconstruct any bad versioned $t$-trajectory from $\sigma_{t+1}$ and the sequence of versioned flaws addressed. We are thus left to count the number of possible sequences of versioned flaws.

Per the hypothesis of Lemma~\ref{lem:master}, each bad $t$-trajectory on $D$ is associated with a rooted labeled witness forest with $t$ vertices such that given the forest we can reconstruct the sequence of flaws addressed along the $t$-trajectory. To count sequences of versioned flaws we relabel the vertices of the witness forest to carry not only the flaw addressed, but also the integer denoting its version (in the corresponding walk on $D^*$). We refer to the resulting object as the versioned  witness forest. Recall that neither the trees, nor the nodes inside each tree in the witness forest are ordered. To facilitate counting we fix an arbitrary ordering $\psi$ of $F$ and map each versioned witness forest  into the unique ordered forest that result by ordering the trees in the forest according to the labels of their roots and similarly ordering the progeny of each vertex according to $\psi$ (recall that both the flaws labeling the roots and the flaws labeling the children of each vertex are distinct). 

Having induced this ordering for the purpose of counting, we will encode each versioned witness forest as a rooted, ordered $d$-ary forest $T$ with exactly $t$ nodes, where $d = \max_{f \in F} \sum_{g \in \Gamma(f)}Z/A_g$ (recall that $Z$ is the least common multiple of the integers $\{A_f: f \in F \}$). In a rooted, ordered $d$-ary forest both the roots and the at most $d$ children of each vertex are ordered. We think of the root of $T$ as having  reserved for each flaw $f \in \mathrm{Roots}(\sigma_1)$ a group of $Z/A_f$ slots, where the $i$-th group of slots corresponds to the $i$-th largest flaw in $F$ according to $\psi$. If $f \in \mathrm{Roots}(\sigma_1)$ is the $i$-th largest flaw in $F$ according to $\psi$ and its version in $V$ is $j$, then we fill the $j$-th slot of the $i$-th group of  slots (recall that the flaws labeling the roots of the witness forest are distinct and that, as a set, belong in the set $\mathrm{Roots}(\sigma_1)$).

Each node $v$ of $T$ corresponds to a node of the witness forest and therefore to a flaw $f$ that was addressed at some point in the $t$-trajectory of the algorithm. Recall now that each node in the witness forest that is labelled by a flaw $f$  has children labelled by distinct flaws in $\Gamma(f)$. We thus think of each node $v$ of $T$ as having precisely $Z/A_g$ slots reserved for each flaw $g \in \Gamma(f)$ (and, thus, at most $d$ reserved slots in total). For each $g \in \Gamma(f)$ whose version is $j$, we fill the $j$-th slot reserved for $g$ and make it a child of $v$ in $T$. Thus, from $T$ we can reconstruct the sequence of versioned flaws addressed with the algorithm. 

At this point we could proceed and bound $|\mathrm{VerBad}(t)|$ by the number of all $d$-ary ordered forests. Indeed, doing so would yield Theorem~\ref{asymmetric}. Such a counting, though, would ignore the fact that the set of flaws labelling the progeny of a node labelled by $f$ is not an arbitrary element of $2^{\Gamma(f)}$ but an element of $\List(f)$. Thus, not every ordered $d$-ary forest is a possible versioned witness forest.  To quantify this observation, we use ideas  from~\cite{PegdenIndepen}. Specifically, we introduce a branching process that produces only ordered $d$-ary forests that correspond to versioned witness forests and bound $|\mathrm{VerBad}(t)|$ by analyzing it.  Before describing the branching process, we introduce some conventions and definitions regarding versions of flaws:
\begin{itemize}

\item For each  $S \subseteq F$ we will denote by $\Versioned(S)$ the set formed by replacing each $f \in S$ by its versions $f_1,f_2,\ldots,f_{Z/A_f}$. For example, $\Versioned(F)$  contains every version of every flaw.

\item For each flaw $f$, we define $\List'(f)$ to be the set that results by replacing each $\{g^1,g^2, \ldots,g^k \} \in \List(f)$ by the $\prod_{i = 1}^{k} Z/A_{g_i}$ sets of the form $\{ g^1_{i_1}, g^2_{i_2}, \ldots, g^k_{i_k} \}$, where $g^i_{i_j}$ is the $i_j$ version of flaw $g^i$.

\item We assign to each flaw $f \in F$ a real number $x_f > 1$.

\item Versioned flaws inherit all the features of the underlying flaw. That is, for each $f_i \in \Versioned(F)$:

\begin{itemize}
\item $x_{f_i} := x_{f}$
\item $\Gamma(f_i) := \Gamma(f)$ 
\item  $\List(f_i) := \List(f)$
\item  $\List'(f_i):= \List'(f)$
\end{itemize}
\end{itemize}

Write $\mathrm{Roots}(\sigma_1) = \mathrm{Roots}$ to simplify notation and let $m = \max_{ S \in \mathrm{Roots}} |S|$. Our branching process takes as input an integer $r \le m$. To start the process we choose an $r$-subset $R$ of $F$ uniformly at random and create $r$ roots, each labeled by a different element of $R$. In each subsequent round, each node $u$ with label $\ell$ ``gives birth" by rejection sampling. Specifically, for each versioned flaw $g_i \in \Versioned(\Gamma(\ell))$ independently, with probability $1/x_{g_i}$ we add a vertex with label $g_i$ as a child of $u$. If the resulting set of children of $u$ is in $\List'(\ell)$ we accept the birth. If not, we delete the children created and try again. Note that while the roots of the resulting trees are labeled by flaws, all other nodes are labeled by versioned flaws. It is not hard to see that this process creates every possible versioned witness forest with $r$ \emph{unversioned} roots with positive probability. Specifically, for a vertex labeled by $\ell$, every set $S \not\in\List'(\ell) $ receives probability 0, while every set $S\in \List'(\ell)$ receives probability proportional to
\[
w_{\ell}(S) = \prod_{g \in S} \frac{1}{x_g} \prod_{h \in \Versioned \left(\Gamma(\ell) \right) \setminus S} \left(1- \frac{1}{x_{h}}\right) \enspace .
\]

To express the exact probability received by each $S\in\List'(\ell)$ we define
\begin{equation}\label{eq:d_def}
Q(S) = \prod_{g\in S} \frac{1}{x_{g}-1}
\end{equation}
and let $Z_{\ell} = \prod_{f \in \Versioned \left( \Gamma(\ell)  \right)}\left(1 - \frac{1}{x_f}\right)$. 
We claim that $w_{\ell}(S) = Z_{\ell} Q(S)$. To see the claim observe that
\[
\frac{ w_{\ell}(S)}
{Z_{\ell}}
=  \frac{ \prod_{g \in S} \frac{1}{x_g} \prod_{h \in \Versioned \left(\Gamma(\ell) \right) \setminus S} \left(1- \frac{1}{x_{h}}\right) }
{\prod_{f \in \Versioned \left( \Gamma(\ell)  \right)}(1 - \frac{1}{x_f}) }  
=  \frac{ \prod_{g \in S} \frac{1}{x_g} }{\prod_{g \in S}(1 - \frac{1}{x_g})   }  
= Q(S) \enspace .
\]
Therefore, each $S\in\List'(\ell)$ receives probability equal to
\begin{equation}\label{eq:sing_birth}
\frac{w_{\ell}(S)}{\sum_{B \in \List'(\ell)} w_{\ell}(B)}
=
\frac
{Q(S) Z_{\ell}}
{\sum_{B \in\List'(\ell)} Q(B) Z_{\ell}}
=\frac{Q(S)}{\sum_{B \in\List'(\ell)} Q(B)}
 \enspace .
\end{equation}

\begin{lemma}\label{branchingLemma}
For any versioned witness forest $\phi$ with set of root-labels $R_{\phi}$, the branching process described above with input $|R_{\phi}|$ produces $\phi$ with probability 
\[
p_{\phi} = \left({\binom{m}{|R_{\phi}|}}  Q(R_{\phi}) \prod_{v \in \phi} \left[(x_v-1)\sum_{S \in\List'(v)} Q(S)\right]  \right)^{-1} \enspace .
\]
\end{lemma}

\begin{proof} 
Let $|R_{\phi}|=r$ and for each node $v$ of $\phi$, let $N(v)$ denote the set of labels of its children. By~\eqref{eq:sing_birth},
\begin{align*}
\frac{1}{p_{\phi}} & =  \binom{m}{r} \prod_{v \in \phi} \frac{\sum_{S \in\List'(v)} Q(S)}{Q(N(v))} \\
& = \binom{m}{r} \, \frac{\prod_{v \in \phi} \sum_{S \in\List'(v)} Q(S)} {\prod_{v \in \phi \setminus R_{\phi}} \frac{1}{x_v -1}}\\
& =  \binom{m}{r} Q(R_{\phi}) \prod_{v \in \phi} \left[(x_v-1)\sum_{S \in\List'(v)} Q(S)\right] \enspace .
\end{align*}
\end{proof}

Let $\mathrm{VWF}(r,t)$ denote the set of versioned witness forests with $r$ unversioned roots and exactly $t$ nodes. Since $\sum_{\phi \in \mathrm{VWF}(r,t)} p_{\phi} \le 1$, it follows that $|\mathrm{VWF}(r,t)| \le \min_{\phi \in \mathrm{VWF}(t)} p_{\phi}^{-1}$ which, by Lemma~\ref{branchingLemma}, equals 
\begin{equation}\label{eq:vbfrt}
\binom{m}{r}
\max_{\phi \in \mathrm{VWF}(r,t)} 
\left\{ Q(R_{\phi}) \prod_{v \in \phi} \left[(x_v-1)\sum_{S \in\List'(v)} Q(S)\right] \right\}\enspace .
\end{equation}

Since for every $f \in F$, each $S \in \List(f)$ gives rise  to $\prod_{g \in S} Z/A_g$ sets in $\List'(f)$, we get~\eqref{multiculti} below. Setting $\mu_f  =  \frac{Z}{A_f  (x_f -1)}>0$ in~\eqref{multiculti} and recalling the definition of $\theta_f$ gives~\eqref{eq:theta_f_bound}. 
\begin{align}
\prod_{v \in \phi} \left[(x_v-1)\sum_{S \in\List'(v)} Q(S)\right] \nonumber 
&\le 
\left(\max_{f \in F} \left[ 
(x_{f} - 1) \sum_{S \in \List'(f) } Q(S)   
\right]\right)^{t} \nonumber \\
&=  
\left(
\max_{f\in F} 
\left[ 
(x_{f} - 1) \sum_{S \in \List(f)} \prod_{g \in S} \frac{Z}{A_g ( x_g - 1)} 
\right]
\right)^{t} \label{multiculti}  \\
&=  
\left(Z
\max_{f\in F} 
\theta_f
\right)^{t}
\enspace . \label{eq:theta_f_bound}
\end{align}
Substituting~\eqref{eq:theta_f_bound} into~\eqref{eq:vbfrt} yields
\[
|\mathrm{VWF}(r,t)| \le (\theta Z)^t
\binom{m}{r}
\max_{\phi \in \mathrm{VWF}(r,t)} 
Q(R_{\phi})  \enspace .
\]

To conclude let $\gamma = \max_{f \in F} Z/A_f$ and let $\mathrm{Roots}(r)$ denote the $r$-subsets in $\mathrm{Roots}$. Assigning all possible version combinations to the roots of each forest in $\mathrm{VWF}(r,t)$ and recalling~\eqref{eq:d_def} se see that
\[
\frac{|\mathrm{VerBad}(t)|}{Z^t} \le \theta^t \left( |\Omega|  \sum_{r = 0}^{m} \binom{m}{r}   \gamma^{r}\max_{ S \in \mathrm{Roots}(r)}   \prod_{f \in S} \frac{\mu_f A_f}{Z} \right)\enspace .
\] 
Recalling that every versioned $t$-trajectory has probability at most $1/Z^{t}$ we see that the binary logarithm of the probability that the walk does not encounter a flawless state within $t$ steps is at most $t  \log_2 \theta + T_0$, where 
\begin{eqnarray*}
 T_0 		& =&  \log_2 | \Omega| +  \log_2 \left( \sum_{r = 0}^{m}  \binom{m}{r} \gamma^{r}  \max_{ S \in \mathrm{Roots}(r)}   \prod_{f \in S} \frac{\mu_f A_f}{Z}   \right) \\
                & \le &  \log_2 | \Omega| + \log_2 \left( \sum_{r = 0}^{m}\binom{m}{r}  \left( \frac{\max_{f \in F}(\mu_f A_f)}{\min_{f \in F} A_f  } \right) ^{r}   \right)  \\
                & = & \log_2 | \Omega | + \max_{ S \in \mathrm{Roots}} |S| \cdot \log_2 \left(1 + \frac{ \max_{f \in F}(\mu_f A_f)}{  \min_{f \in F} A_f} \right) \enspace .
\end{eqnarray*}

Therefore, if $t = (T_0 + s) / \log_2 (1/ \theta) \le (T_0 + s) / \delta$, the probability that the  random walk on $D^*$, and therefore on $D$, does not reach a flawless state within $t$ steps is at most  $2^{-s}$.

\section{A First Application -  Hamilton Cycles in Hypergraphs}\label{sec:ham}

\subsection{Preliminaries}

An (edge) coloring of a hypergraph $H(V,E)$ is a function $\phi: E \rightarrow \mathbb{N}$ assigning natural numbers (colors) to the edges of $H$. A hypergraph $H$ together with a given coloring $\phi$ will be dubbed a \emph{colored hypergraph}. We will say that $e_1 \neq e_2 \in E$ are \emph{adjacent} if $e_1 \cap e_2 \neq \emptyset$. A subhypergraph $S$ of a colored hypergraph $H$ is said to be \emph{properly} colored if every two adjacent edges of $S$ receive different colors. If, further, every edge of $S$ receives a different color, i.e., if $\phi$ is injective on $S$, we will say that $S$ is \emph{rainbow}. 

To promote the presence of properly colored and rainbow subhypergraphs we introduce the following restrictions on hypergraph colorings. For a coloring $\phi$ and a color $i \in \mathbb{N}$, let $H_{\phi}^i = H[\phi^{-1}(i)]$ denote the hypergraph induced by the edges of color $i$ in $\phi$. We say that $\phi$ is \emph{$r$-degree bounded} if $H_{\phi}^i$ has maximum degree at most $r$, for all $i \in \mathbb{N}$. If $H_{\phi}^i$ has at most $r$ edges, for all $i \in \mathbb{N}$,  we say that $\phi$ is \emph{$r$-bounded}.

We investigate the existence of properly colored and rainbow Hamilton cycles in colored $k$-uniform complete hypergraphs, $k \ge 3$. (A hypergraph is $k$-\emph{uniform} if every edge has size $k$; it is complete if all $k$-element subsets of the vertices form edges).  For $1 \le \ell < k$, an $\ell$\emph{-overlapping cycle} is a $k$-uniform hypergraph in which, for some cyclic ordering of its vertices, every edge consists of $k$ consecutive vertices (in the cyclic ordering), and every two consecutive edges (in the natural ordering of the edges induced by the ordering of the vertices) share exactly $\ell$ vertices. Thus, the number of edges in an $\ell$-overlapping cycle with $s$ vertices is $\lfloor s /(k- \ell)\rfloor$. The two extreme cases $\ell = 1$ and $\ell = k-1$ are referred to as, respectively, \emph{loose} and \emph{tight} cycles.  
\begin{remark}\label{too_funny}
A tight cycle on $s$ vertices contains an $\ell$-overlapping cycle on the same vertex set (with the same cyclic ordering), whenever $k-\ell$ divides $s$.
\end{remark}

Given a $k$-uniform hypergraph $H$ on $n$ vertices where $k-\ell$ divides $n$, an $\ell$-overlapping cycle is called \emph{Hamilton} if it goes through every vertex of $H$, that is, if $s = n $. We denote such a Hamilton cycle by $C_{n}^{(k)}(\ell)$. Let $K_n^{(k)}$ denote the complete $k$-uniform hypergraph on $n$ vertices. 

In~\cite{Hamilton1}, Dudek, Frieze and Ruci\'{n}ski proved the following.
\begin{theorem}[\cite{Hamilton1}]\label{Rainbow}
For every $1 \le \ell < k $ there is a constant $c = c(k, \ell)$ such that if $n$ is sufficiently large and $k - \ell$ divides $n$, then any $cn^{k-\ell}$-bounded coloring of $K_{n}^{(k)}$ contains a rainbow copy of $C_{n}^{(k)}(\ell)$. 
\end{theorem}

\begin{theorem}[\cite{Hamilton1}]\label{Proper}
For every $1 \le \ell < k$ there is a constant $d = d(k,\ell)$ such that if $n$ is sufficiently large and $k - \ell$ divides $n$, then any $dn^{k-\ell}$-degree bounded coloring of $K_{n}^{(k)}$ contains a properly \mbox{colored copy of $C_{n}^{(k)}(\ell)$.}
\end{theorem}

In~\cite{Hamilton2}, Dudek and Ferrara strengthened Theorems~\ref{Rainbow}, \ref{Proper} as follows. Say that a coloring is \emph{$(a,r)$-bounded} if for each color $i$, every set of $a$ vertices is contained in at most $r$ edges of color $i$. An $r$-degree bounded coloring is, thus, $(1,r)$-bounded and an $r$-bounded coloring is $(0,r)$-bounded (as it has at most $r$ edges of color $i$). Thus, Theorem~\ref{Rainbow} follows from Theorem~\ref{Rainbow2} since for every $1 \le \ell \le k$, every $cn^{k-\ell}$-bounded coloring is both $(0, cn^{k-1})$-bounded and $(\ell,cn^{k-\ell})$-bounded. Similarly, Theorem~\ref{Proper} follows from Theorem~\ref{Proper2} since for every $1 \le \ell \le k$, every $dn^{k-\ell}$-degree-bounded coloring is $(\ell,dn^{k-\ell})$-bounded. 

\begin{theorem}[\cite{Hamilton2}]\label{Rainbow2}
For every $1 \le \ell < k$ there is a constant $c = c(k,\ell)$ such \mbox{that if $n$ is sufficiently large and} $k - \ell$ divides $n$, then any $(\ell,cn^{k-\ell})$-bounded coloring of $K_{n}^{(k)}$ that is $(0, cn^{k-1})$-bounded  contains a rainbow copy of $C_{n}^{(k)}(\ell)$. 
\end{theorem}

\begin{theorem}[\cite{Hamilton2}]\label{Proper2}
For every $1 \le \ell < k$ there is a constant $d = d(k,\ell)$ such that if $n$ is sufficiently large and $k-\ell$ divides $n$, then any $(\ell,d n^{k-\ell})$-bounded coloring of $K_{n}^{(k)}$ contains a properly colored of $C_{n}^{(k)}(\ell)$.
\end{theorem}

We make Theorems~\ref{Rainbow2} and~\ref{Proper2} constructive while also improving the constants from~\cite{Hamilton2}.
\begin{theorem}\label{ourfirstapp}
The Hamilton cycles guaranteed by Theorems~\ref{Rainbow2}, \ref{Proper2} can be found in time $O(n^{4k})$. 
\end{theorem}

\subsection{Proof of Theorem~\ref{ourfirstapp}} 

We will use the following proposition, which can be derived easily by synthesizing results from~\cite{Hamilton1}.

\begin{proposition}\label{TheWholePoint} 
Fix $1 \le \ell < k$. Let $\{e,f\}$ be any pair of edges of $K_{n}^{(k)}$ with $|e \cap f| = \alpha \le \ell$. Let $X$ be any set of pairs $\{g,h\}$ of edges of $K_{n}^{(k)}$ satisfying $(e \cup f) \cap (g \cup h)  = \emptyset $. 
\begin{itemize}
\item 
Let $\mathcal{C}(X)$ be the set of all copies $C$ of $C_{n}^{(k)}(k-1)$ in $K_{n}^{(k)}$ such that $\{g,h\} \nsubseteq C$ for all $\{g,h\} \in X$.
\item 
Let $\mathcal{C}_{e,f}(X) = \{C \in \mathcal{C}(X) : \{e,f\} \subset C \}$.
\end{itemize}

There is $\delta = \delta(k, \ell) >0$ such that if $\mathcal{C}_{e,f}(X) \ne \emptyset$, one can find a disjoint family $\{ \mathcal{S}_{C} : C \in \mathcal{C}_{e,f}(X) \}$ of sets of copies of $C_{n}^{(k)}(k-1)$ from $\mathcal{C}(X)$ (indexed by the copies $C \in \mathcal{C}_{e,f}(X)$) such that for all $C \in \mathcal{C}_{e,f}(X)$:
\begin{enumerate}
\item\label{mikos}   
$\mathcal{S}_{C} \cap \mathcal{C}_{e,f}(X) = \emptyset$.
\item\label{otinane}
$| \mathcal{S}_{C}| \ge \delta n^{2k- 2}$, if $\alpha  = 0$.
\item\label{tothirio}
$| \mathcal{S}_{C}| \ge \delta n^{2k- \alpha -1}$, if $1 \le \alpha \le \ell$.
\end{enumerate}
Furthermore, a uniformly random element of each set $\mathcal{S}_{C}$ can be sampled in time $O(n^{2k})$.
\end{proposition}

\begin{proof}[Constructive Proof of Theorem~\ref{Rainbow2}]
Fix $ 1 \le \ell < k $ and let $\phi$ be a coloring of $K_{n}^{(k)}$. Define the set $M$ to consist of all pairs of edges that have the same color and share at most $\ell$ vertices, i.e.,
\[
M = \{\{e_1, e_2\}: e_1, e_2 \in K_{n}^{(k)}, \phi(e_1) = \phi(e_2) , \text{ and }
| e_1 \cap e_2| \le \ell \} \enspace .
\]

Let $\Omega$ be the set of copies of $C_{n}^{(k)}(k-1)$ in $K_{n}^{(k)}$. For each pair of edges $\{e,f\} \in M$ we define the flaw
\[
F_{e,f} = \{ C \subset K_{n}^{(k)}: C \sim C_{n}^{(k)}(k-1) \text{ and } \{e,f\} \subset C\} \enspace .
\]
That is, $F_{e,f}$ consists of all tight Hamilton cycles that contain both $e$ and $f$ (and are thus improperly colored since $\phi(e) = \phi(f)$). A flawless  $C \in \Omega$ is, thus, a tight Hamilton cycle whose edges have distinct colors. Since $k - \ell$ divides $n$, Remark~\ref{too_funny} implies that any such cycle $C$ contains a rainbow copy of $C_{n}^{(k)}(\ell)$.

Having defined flaws, we now need to define actions for each flaw. To that end, for each pair $\{e,f\} \in M$ and each integer $0 \le \alpha \le \ell$, we define
\[
Y_{e,f}(\alpha) = \{\{e',f'\} \in M: \{e',f'\} \ne \{e,f\}, |e' \cap f'| = \alpha \text{ and } (e \cup f)\cap( e' \cup f') \ne \emptyset \} \enspace .
\]
Let
\[
Y_{e,f} = \bigcup_{ \alpha = 0}^{\ell} Y_{e,f}(\alpha)		\qquad \text{and} \qquad X_{e,f} = M \setminus( Y_{e,f} \cup \{e,f \} ) \enspace .
\]
For each Hamilton cycle $C \in \Omega$, for each pair of edges $\{e,f\} \in M$ such that flaw $F_{e,f}$ is present in $C$, we invoke Proposition~\ref{TheWholePoint} with $e,f$ and $X = X_{e,f}$. Let $\mathcal{S}_C$ be the set of Hamilton cycles guaranteed by Proposition~\ref{TheWholePoint}. We let $A(F_{e,f},C) = \mathcal{S}_{C}$. To lighten notation, we let  $A_{e,f} := A_{F_{e,f}} = \min_{C  \in \Omega} |A(F_{e,f},C)|$. By Proposition~\ref{TheWholePoint} we thus have:
\begin{itemize}
\item 
$D$ is atomic since for each flaw we have a disjoint family of sets of cycles (actions).
\item 
If $|e \cap f| = \alpha$, then
\[
A_{e,f}  \ge  
\begin{cases}
		\delta n^{2k-2} 				& \text{ if } \alpha = 0 , 
        	\cr
       	\delta n^{2k-\alpha-1} 	& \text{ if } 1\le \alpha \le \ell \enspace .
\end{cases}
\]
\item 
If $F_{g,h} \in \Gamma(F_{e,f})$ then $\{g,h\} \in Y_{e,f}$ since $\{g,h\} \notin X_{e,f}$ and $\{g,h\} \ne \{e,f\}$
\end{itemize}

To bound $|Y_{e,f}|$ we use the following fact, established in~\cite{Hamilton2}. For every $c>0$, if $\phi$ is $(\ell, cn^{k-\ell})$-bounded and $(0, cn^{k-1})$-bounded, then there exists $n_0 = n_0(c)$ such that for all $n\ge n_0$,
\[
\max_{\{e,f\} \in M} |Y_{e,f}(\alpha)| 
\le
\begin{cases}
		2ckn^{2k-2} 				& \text{ if } \alpha = 0 , 
        	\cr
       	2ck^{\ell+1}n^{2k-\alpha-1} 	& \text{ if } 1\le \alpha \le \ell \enspace .
\end{cases}
\]

Therefore, if $c = \delta (2\mathrm{e}k(1+\ell k^{\ell}))^{-1}$,  for each pair of edges $\{e,f\} \in M$ with $|e \cap f| = \alpha$, we have
\[
\sum_{F_{g,h} \in \Gamma(F_{e,f})} \frac{1}{A_{g,h}} 
=
\sum_{\alpha=0}^{\ell} \sum_{\{g,h\} \in Y_{e,f}(\alpha)} \frac{1}{A_{g,h}} 
\le 
\frac{2ckn^{2k-2}}{\delta n^{2k-2}} + \sum_{\alpha = 1}^{\ell} \frac{2ck^{\ell+1}n^{2k-\alpha-1}}{ \delta n^{2k - \alpha -1}} 
=
\frac{2ck}{\delta} (1+ \ell k^{\ell}) 
< 
\frac{1}{\mathrm{e}} \enspace ,
\]
where $\mathrm{e}$ is Euler's constant and the last inequality holds for every $\ell \ge 0$. 

By Theorem~\ref{cor:simple}, the uniform random walk on $D$ terminates after $O( |M| + \log | \Omega | ))$ steps with high probability, where $\log_2 |\Omega| \le \log_2 \binom {n} {k} ^n \le nk \log_2 n  $ and $|M| \le n^{2k}$. As in each step  we need $O(n^{2k})$ time to find the greatest flaw and $O(n^{2k})$ time to choose an action for it, we have proven the theorem.

\end{proof}

\begin{proof}[Constructive Proof of Theorem~\ref{Proper2}] 
Fix $1 \le \ell < k$ and let $\phi$ be a coloring of $K_{n}^{(k)}$. The proof goes along the lines of the proof of Theorem~\ref{Rainbow2}, except we slightly modify the definition of the set $M$ from the in that it contains no pair of disjoint edges, i.e., 
\[
M = \{ \{e_1,e_2\}: e_1,e_2 \in K_n^{(k)}, 1 \le | e_1 \cap e_2 | \le \ell , \text{ and } \phi(e_1) = \phi(e_2) \} \enspace .
\]
We change $M$ in this way since a properly colored cycle may contain nonadjacent edges of the same color. As before, for each pair of edges $\{e,f\} \in M$ define the flaw
\[
F_{e,f} = \{ C \subset K_{n}^{(k)}: C \sim C_{n}^{(k)}(k-1) \text{ and } \{e,f\} \subset C\} \enspace .
\]

A flawless $C \in \Omega$ is, thus, a tight Hamilton cycle such that for every pair of its edges $e$ and $f$ with $1 \le | e \cap f | \le \ell$ we have $\phi(e) \ne \phi(f)$. Since $k - \ell$ divides $n$, once again Remark ~\ref{too_funny} implies that any such cycle $C$ contains properly colored copy of $C_{n}^{(k)}(\ell)$. 

We define the sets $Y_{e,f}(\alpha)$ only for $\alpha \in [\ell]$  and consequently $Y_{e,f} = \cup_{a =1}^{\ell} Y_{e,f}(\alpha)$. 
An argument identical to that in the proof of Theorem~\ref{Rainbow2} shows that if $d = \delta( 2\mathrm{e}\ell k^{\ell+1})^{-1}$,  Theorem~\ref{cor:simple} applies.
\end{proof}

\section{A Second Application -  Rainbow Matchings in Complete  Graphs}\label{sec:match}

In an edge-colored graph $G=(V,E)$, say that $S \subseteq E$ is \emph{rainbow} if its elements have distinct colors.
\begin{theorem}\label{Matchings}
For any $C > 2\mathrm{e}$, given any edge-coloring of the complete graph on $2n$ vertices in which each color appears on at most $n/C$ edges a rainbow perfect matching can be found in $O(n^4 \log n)$ time.
\end{theorem}

\begin{proof}
Let $\phi$ be any edge-coloring of $K_{2n}$ in which each color appears on at most $q$ edges. Let $P=P(\phi)$ be the set of all pairs of vertex-disjoint edges with the same color in $\phi$, i.e., $P = \{ \{e_1,e_2 \}: \phi(e_1) = \phi(e_2) \}$. 
Let $\Omega$ be the set of all perfect matchings of $K_{2n}$. For each $\{e_i,e_j\} \in P$ let
\[
f_{i,j} = \{ M \in \Omega:  \{e_i,e_j\} \subset M\} \enspace .
\]
Thus, an element of $\Omega$ is flawless iff it is a rainbow perfect matching. 

To address the flaw induced by edges $e_i,e_j \in M$ we select two other edges $e_k, e_{\ell} \in M$ and in each of the two edge-pairs $\{e_i,e_k\}$ and $\{e_j,e_{\ell}\}$ we select one of the other two matchings. More precisely, let $f=\{\{v_1,v_2\},\{v_3,v_4\}\} \in P$ and assume, without loss of generality,  that $v_1>v_2$, $v_3>v_4$, and $v_1 > v_3$. For $M \in f$, the set $A(f,M)$ consists of all possible outputs of {\sc{Switch}}$(M,\{v_1,v_2\},\{v_3,v_4\})$.

\begin{algorithm}[H] 
\caption{ {\sc{Switch}}$(M,\{v_1,v_2\},\{v_3,v_4\})$}
\begin{algorithmic}[1]

\State Let $u_1$ be any vertex other than $\{v_1,v_2,v_3,v_4\}$. \label{v1part}

\State Let $u_2$ be the vertex such that $(u_1,u_2) \in M$.

\State Let $u_3$ be any vertex other than $\{v_1,v_2,v_3,v_4,u_1,u_2\}$.\label{v3part}

\State Let $u_4$ be the vertex such that $(u_3,u_4) \in M$.

\State Output $M' \in \Omega$ by removing from $M$ edges $(v_1,v_2)$, $(u_1,u_2)$, $(v_3,v_4)$, $(u_3,u_4)$, and adding edges $(v_1,u_1)$, $(v_2,u_2)$, $(v_3,u_3)$,$(v_4,u_4)$.  
\end{algorithmic}
\end{algorithm}

Enumerating the choices in Steps~\ref{v1part} and~\ref{v3part} we see that $|A(f,M)| = (2n-4)(2n-6)$. On the other hand, each of the four edges inserted by each action of each set $A(f,\cdot)$ has exactly one vertex in $V = \{v_1,\ldots,v_4\}$ and one vertex outside $V$, and can only form a flaw along with an edge having the same color as itself. Therefore, $|\Gamma (f_{i,j})| \le 4 (2n-4)(q-1)$ for every $i,j$. 

Let $D$ be the directed graph on $\Omega$ corresponding to the above sets of actions. To prove atomicity consider any arc $M \xrightarrow{f} M'$. Adding to $M'$ the two edges defining $f$ yields two edge-disjoint paths of length 3. Considering the 4-cycle closing each path, $M$ results by taking in each cycle the 2-matching containing the edges of the flaw. Let $A_{i,j} := A_{f_{i,j}} = \min_{M \in f} | A(f,M)|$. To conclude we note that for every $C>2\mathrm{e}$, there exists $\delta = \delta(C)$, such that for all $n \ge n_0(C)$,
\[
\sum_{f_{k,\ell} \in \Gamma(f_{i,j})} \frac{1}{A_{k,\ell}} 
\le\frac{4(2n-4)(q-1)}{(2n-4)(2n-6)} = \frac{2(q-1)}{n-3} \le
\frac{1}{\mathrm{e}} - \delta \enspace .
\]

To bound the running time let $q_i$ be the number of edges with color $i$. Trivially, $|P| \le \sum_i \binom{q_i}{2}<n^4$, and it is not hard to see that in fact $|P| < n^3$. At the same time, $|\Omega| = (2n-1)!!$ implying $\log_2|\Omega| = O(n \log n)$. By  Theorem~\ref{cor:simple}, the uniform random walk on $D$ terminates after $O(|P| + \log_2 |\Omega|)$ steps with high probability. By fixing an arbitrary ordering of the colors and sorting the edges in the current matching according to color after each step we can trivially find the greatest flaw present in time $O(n \log n)$ (flaws are ordered first by color and then by the lexicographic order of the four vertices involved) .
\end{proof}

\section{Color-Blind index of Graphs - When Directed Causality Matters} \label{sec:cb}

\subsection{Preliminaries}

Let $\phi: E(G) \rightarrow [k]$ be an edge-coloring, not necessarily proper, of a graph $G(V,E)$. For each vertex $v \in V$, let $c(v) = (a_1, \ldots, a_k)$, where $a_i = |\{u: \{u,v\} \in E, c(\{u,v\}) = i|$, for $i \in [k]$. We say that $\phi$ \emph{distinguishes neighbors by multisets} if $c$ is a proper vertex-coloring of $V$ and denote by $\mathrm{ndi}_\mathrm{m}(G)$ the smallest  $k$ for which such $c$ exists. Clearly, if $G$ contains $K_2$ as a connected component, no such edge-coloring exists. Addario-Berry et al.~\cite{AddarioBerry} proved that as long as that is not the case,  $\mathrm{ndi}_\mathrm{m}(G) \le 4$.

Kalinowski et al.~\cite{blind} introduced a fascinating twist to the above concept that captures color-blindness. A color-blind person looking at two green edges and one red edge sees two edges of the same color and one edge of another color. And their view would remain the same if, instead, we had two red edges and one green. If we re-order the sequence $c(v) = ( a_1, a_2, \ldots, a_k)$ non-decreasingly, we obtain a sequence $p(v) = (d_1, \ldots, d_k)$, called the \emph{palette} of vertex $v$. (Thus, there is a bijection between the set of all possible palettes of a vertex $v$ of degree $d$ and the set of all partitions of the integer $d$ into at most $k$ parts). We say that a \emph{color-blind person can distinguish neighbors}  if $p(u) \ne p(v)$ for every edge $\{u,v\} \in E$, i.e., if $p$ is a proper coloring of the vertices of $G$. The smallest possible number $k$ for which such an edge-coloring exists is called the \emph{color-blind index} of a graph $G$ and is denoted by $\mathrm{dal}(G)$, the notation refering to the English chemist John Dalton who was the first scientist to take academic interest in the subject of color blindness.

It has to be noted that there are infinitely many graphs, e.g., odd cycles, for which the color-blind index is not defined. In~\cite{blind} it was conjectured that there exists a number $K$ such that $\mathrm{dal}(G) \le K$, for every graph $G$ for which $\mathrm{dal}(G)$ is defined. The authors prove this conjecture for complete graphs, regular bipartite graphs, regular graphs of sufficiently large degree, and  graphs with bounded ratio $\Delta(G)/\delta(G)$. 
\begin{theorem}[\cite{blind}]\label{Irregular}
For every $R \ge 1$, there exists $\delta_0=\delta_0(R)$ such that if $\delta(G) \ge \delta_0$ and $\Delta(G) \le R \delta(G)$, then $\mathrm{dal}(G) \le 6$.
\end{theorem}
Theorem~\ref{Irregular}, covering both regular and irregular graphs, is proven in~\cite{blind} by applying the lopsided LLL with a \emph{directed} lopsidepency graph. Thus, the result does not fit either the variable framework of~\cite{MT}, or the permutation setting of~\cite{SrinivasanPerm}.

\begin{theorem}
The colorings guaranteed by Theorem~\ref{Irregular} can be found in $O \left(|E(G)|\left( 1 + \Delta(G) - \delta(G) \right) \right)$ time.
\end{theorem}

\subsection{Proof}

Let $\Omega$ be the set of all edge-colorings, not necessarily valid, of $G(V,E)$ with 6 colors. Let $S_v$ denote the set of edges incident to a vertex $v \in V$, and for $\{u,v\} \in E$, let $S_{u,v} = S_u \cup S_v$. Fix an arbitrary ordering of $V$. For each edge $\{u,v\}$ where $u < v$ and $d(u) = d(v)$, we define a set of flaws as follows.

Let $\mathrm{Bad}(u,v)$ be the set of all edge colorings of $G$ with $p(u) = p(v)$. Partition $\mathrm{Bad}(u,v)$ into equivalence classes, forming a partition $P$, where two colorings are equivalent if they agree on the coloring of $S_u \setminus \{u,v \}$. Further partition each class $C \in P$ into equivalence classes, forming a partition $Q(C)$,  where two colorings in $C$ belong in the same equivalence class, if they agree on the coloring of $S_v$.

We claim that if  $d(u) = d(v) = d$, then for each $C \in P$, the size of $Q(C)$, i.e the number of equivalence classes in $Q(C)$, is at most 
 \begin{equation}\label{even_funnier}
f(d):= \max_{d_1+\cdots+d_{6} = d} \binom{d}{d_1,d_2,\ldots,d_6} 6!\enspace .
\end{equation}
To see this observe that for any coloring of $S_u \setminus \{u,v\}$ there exist numbers $d_1 \ge d_2 \ge \cdots \ge d_6$ summing to $d$ such that $p(u) = p(v)$ implies $p(v)=(d_1,d_2,\ldots,d_6)$. Therefore, the number of elements in $Q(C)$ is bounded by the number ways to partition the $d$ edges in $S_v$ into sets of sizes $d_1,\ldots,d_6$ times the 6! ways of assigning distinct colors to the sets. Finally, it is not hard to see that~\eqref{even_funnier} is maximized when $|d_i-d_j|\le 1$ for all $i,j\in [6]$, implying 
\begin{equation}\label{etoimo}
f(d)= \max_{d_1+\cdots+d_{6} = d} \binom{d}{d_1,d_2,\ldots,d_6} 6! \le 6^d \frac{27  \sqrt{2}}{(\pi d)^{5/2}}  6!  < 1572 \frac{6^d}{d^{5/2}}\enspace .
\end{equation}

For each $C \in P$, consider an arbitrary ordering for the members of $Q(C)$ and let $Z_i(C)$, $i \in [f(d)]$, be the $i$-th member of $Q(C)$. For each $i \in [f(d)]$ define the flaw $f_{u,v}^{i}$ as :

\begin{align*}
f_{u,v}^{i} = \bigcup_{C \in P} Z_i(C)
\end{align*}
Thus, a flawless element is an element where there is no edge $(u,v)$ such that $p(u) = p(v)$, as we wanted.
For each flaw $f_{u,v}^{i}$, where $u < v$, for each $\phi \in f_{u,v}^{i}$, the set of actions $A(f_{u,v}^{i},\phi) \subseteq \Omega$ consists of all possible recolorings of $S_v$ in $\phi$. Thus, $|A(f_{u,v}^{i},\phi)| = 6^d$, for all $u,v,i$. 
 
Let $D$ be the directed graph on $\Omega$ corresponding to these actions. To establish the atomicity of $D$ it suffices to show that for every transition $\phi \xrightarrow{f_{u,v}^{i}} \phi'$, where $u < v$, if we are given $\phi'$ and $f_{u,v}^{i}$, we can reconstruct $\phi$. To see this, at first notice that $\phi$ and $\phi'$ differ only in the coloring of $S_v$. Therefore, $\phi'$ implies the coloring $C$ of $S_u \setminus \{u,v\}$ in $\phi$, while the integer $i$ implies $Z_i(C)$ and therefore the coloring of $S_v$.

Fix $\{u,v\} \in E$ with $u < v$ and $d(u) = d(v)$ and let  $M(u,v) = \bigcup_{e \in S_v} S_e$. That is, $M(u,v)$ is the set of edges that are adjacent to $u$ or $v$, or to edges adjacent to $v$. Observe that any action taken at a state $\phi$ to address flaw $f^{i}_{u,v}$ only introduces flaws that are associated with edges in $M(u,v)$.  To see this, notice that when we recolor an edge we only introduce flaws associated with edges adjacent to it.

Recall by our discussion above that for any edge whose endpoints have degree $d$, the number of flaws associated with it
is bounded by $f(d)$ and that for any such flaw $g$, we have $A_g := \min_{\phi \in g} |A(g,\phi)| = 6^d$. Therefore, for any such flaw $f_{u,v}^i$ (and, thus, for every flaw),
\begin{equation}
\sum_{g \in \Gamma(f_{u,v}^{i})} \frac{1}{A_{g}}  \le      \sum_{e \in M(u,v)  }  \frac{f(d_e)}{6^{d_e} } <   |M(u,v)|   \cdot \frac{1572}{\delta(G)^{5/2}} \le 1572 \frac{\Delta(G)^2 }{ \delta(G)^{5/2} } \enspace . \label{teleutaia}
\end{equation}

From~\eqref{teleutaia} we see that if $\delta(G) > \delta_0 = (1572 \mathrm{e} R^2 )^{{2}}$, then the condition of Theorem~\ref{cor:simple} holds. Regarding the running time, we note that $\log_2 |\Omega| =  \log_2 6^{|E|} = O ( |E|)$ while, clearly, $|U(\sigma_1)| \le |E|$ since in each state each edge can ``give rise" to at most one flaw.

\section{ Latin Transversals - A case of Dense  Neighborhoods}\label{sec:latin}

Let $M$ be an $n \times n$ matrix whose entries come from a set of colors $C$. A \emph{Latin Transversal} of $M$ is a permutation $ \pi \in S_n$ such that the entries $\{M(i,\pi(i))\}_{i=1}^n$ have distinct colors, i.e., a selection of $n$ entries in distinct rows and columns such that no two elements have the same color.

\begin{theorem}
If each color $c \in C$ appears at most $\Delta \le \frac{27}{256} n$ times in $M$, then the Recursive Walk will find a Latin Transversal of $M$ in $O( n \log n)$ steps with high probability.
\end{theorem}

\begin{proof}

Let $M$ be any matrix in which each color appears at most $\Delta$ times and let $\Omega = S_n$ be  the set of all permutations of $[n]$. Let $P = P(M)$ be the set of all quadruples ($i,j,i',j')$ such that $M(i,j) = M(i',j')$. For each quadruple $(i,j,i',j')\in P$ let
\[
f_{i,j,i',j'}= \{ \pi \in \Omega: \pi(i)= j \text{ and } \pi(i') = j'  \} \enspace .
\]
Thus, an element of $\Omega$ is flawless iff it is a Latin Transversal of $M$.
 
To address the flaw induced by a pair of entries $(i,j)$, $(i',j')$ of $M$ in an element $\pi \in \Omega$, we select two other entries $(\alpha, \beta)$, $(\alpha',\beta')$, also selected by $\pi$, and replace the four entries $(i,j),(i',j'),(\alpha, \beta),(\alpha',\beta')$ with the four entries $(i,\beta)$, $(i',\beta')$, $(\alpha,j)$, and $(\alpha',j')$. More precisely, for $\pi \in f = f_{i,j,i',j'}$ the set $A(f,\pi)$ consists of all possible outputs  of {\sc{Switch}}$(\pi,i,j,i',j')$.

\begin{algorithm}[H] 
\caption{ {\sc{Switch}}$(\pi,i,j,i',j')$}
\begin{algorithmic}[1]
\State Let $\alpha$ be any element of $[n]$. Let $\beta = \pi(\alpha)$. \label{v1partlatin}
\State Let $\alpha'\neq \alpha$ be any element of $[n]$. Let $\beta'=\pi(\alpha')$. \label{v2partlatin}

\State Modify $\pi$ to $\rho$ by the following ``switch": $\rho(i) = \beta$, $\rho(i') = \beta'$, $\rho(\alpha) = j$, $\rho(\alpha') = j'$.

\end{algorithmic}
\end{algorithm}

To prove atomicity consider any action $\pi \xrightarrow{f_{i,j,i',j'} } \rho$. Suppose that $\rho(i) = \beta$, $\rho(i') = \beta'$, $\rho^{-1}(j) = \alpha$, and $\rho^{-1}(j') = \alpha'$. Given $\rho$ and $(i,j,i',j')$, we see that the image of every element under $\pi$ other than $i,i',\alpha, \alpha'$ is the same as under $\rho$, while $\pi(i) = j$, $\pi(i') = j'$, $\pi(\alpha) = \beta$ and $\pi(\alpha') = \beta'$.

Enumerating the choices in Steps~\ref{v1partlatin} and~\ref{v2partlatin} we see that $|A(f, \pi) | = n (n-1)$. Let us now consider the form of the causality graph $C$. It is not hard to see that if $f \rightarrow g$ exists in $C$, then $g \rightarrow f$ exists as well, so we will think of the undirected version $G$ of $C$. Two  flaws $f_{i,j,i',j'}$ and $f_{p,q,p',q'}$ are adjacent in $G$ if and only if $\{i,i' \} \cap \{p, p' \} \ne \emptyset$ or $\{j,j'\} \cap \{q,q'\} \ne \emptyset$. Thus, each flaw $f_{i,j,i',j'}$ is adjacent to four types of flaws, corresponding to the four new entries $(i,\beta)$, $(i',\beta')$, $(\alpha,i)$, and $(\alpha',j')$. The maximum degree of $G$ is at most $4n(\Delta-1)$ since for a fixed $(i, j, i', j')$ we can choose $(s, t)$ with $s \in \{i, i'\}$ or $t \in \{j, j'\}$ in $4n$ different ways and, as there are at most $\Delta$ entries of $M$ with any given color, once $(s, t)$ has been chosen there are at most $\Delta-1$ choices for $(s', t')$ such that $M(s,t) = M(s',t')$. Thus, the set of vertices in $\Gamma(f_{i, j, i', j'})$ is the union of four subsets, each of cardinality at most $n(\Delta-1)$, where crucially the vertices in each subset form a clique. 
 
Setting $\mu_f =  \mu$ for each flaw $f$,  the condition~\eqref{eq:mc3} of Theorem~\ref{olala} becomes
\begin{eqnarray}
 n(n-1)  & \ge & 
 \mu^{-1} \sum_{i=0}^{4}\binom{4}{i} \left(n \left(\Delta-1 \right)\right)^{i}  \mu^i = 
 \mu^{-1} \left(1+\mu n \left(\Delta-1 \right)\right)^{4}
 \enspace .\label{eq_late}
\end{eqnarray}
It is easy to see that if $\mu = \frac{1}{3n(\Delta-1)}$, then~\eqref{eq_late} holds for all $ \Delta \le \frac{27}{256}n$.

To bound the running time notice that for every state $\sigma_1$ the largest independent subset of $U(\sigma_1)$ is of size $O(n)$ and that  $ \log_2 |\Omega | =  \log_2 n! = \Theta( n \log n)$.

\end{proof}

%\section*{Acknowlegements}
%
%DA thanks Christos Papadimitriou for introducing him to the Probabilistic Method a quarter century ago and Mario Szegedy for inspiration to shun the lowest common denominator ($Z$). 

\bibliographystyle{plain}
\bibliography{../smoser}

\end{document}